\documentclass[11pt]{article}

\usepackage{xcolor}
\usepackage[letterpaper, margin=1in]{geometry}
\definecolor{gray}{rgb}{0.3, 0.3, 0.8}
\definecolor{light-gray}{gray}{0.95}
\definecolor{darkgreen}{rgb}{0.0,0.6,0.0}
\usepackage[colorlinks=true,urlcolor=red, linkcolor=darkgreen,citecolor=red]{hyperref}
\usepackage{graphicx}
\usepackage{subcaption}
\usepackage{amsmath,amsfonts,amsthm}
\usepackage{latexsym,amssymb}
\usepackage{mathtools}
\usepackage[boxruled]{algorithm2e}
\usepackage{algorithmic}
\usepackage{thmtools,thm-restate}
\usepackage{enumitem}
\usepackage{lmodern}
\usepackage{bbm}
\usepackage{bm}

\usepackage{verbatim}

\AtBeginDocument{\let\phi\varphi}
\newcommand{\primal}{\bm{\gamma}}
\newcommand{\realdual}{\bm{\gamma^*}}
\newcommand{\dual}{\bm{\delta}}
\newcommand{\ent}{\dual^\mathbf{Ent}}
\newcommand{\Z}{\mathbb{Z}}
\newcommand{\R}{\mathbb{R}}

\newcommand{\BP}{\mathbb{P}}
\newcommand{\BE}{\mathbb{E}}
\newcommand{\BR}{\mathbb{R}}
\newcommand{\RR}{\mathbb{R}}

\newcommand{\eps}{\mathsf{r}}

\newcommand{\CA}{\mathcal A}
\newcommand{\CE}{\mathcal E}

\newcommand{\CCT}{\mathcal C}

\newcommand{\CS}{\mathcal S}

\newcommand{\CT}{\mathcal T}

\newcommand{\CP}{\mathcal{P}}

\newcommand{\CX}{\mathcal X}
\newcommand{\CY}{\mathcal Y}
\newcommand{\E}{\mathbb{E}}
\newcommand{\supp}{\mathsf{supp}}

\newcommand{\ind}{\mathbf{1}}

\newcommand{\diam}{\mathsf{diam}}

\newcommand{\valch}{\mathsf{val}_h}
\newcommand{\valsep}{\mathsf{val}_h}
\DeclareMathOperator*{\argmax}{arg\,max}
\DeclareMathOperator*{\argmin}{arg\,min}
\newcommand{\taucov}{\tau_{\text{cov}}}
\newcommand{\tcov}{t_{\text{cov}}}
\newcommand{\LN}{\mathcal L}
\newcommand{\val}{\mathsf{val}}

\newcommand{\depth}{\mathsf{depth}}
\newcommand{\N}{\mathbb{N}}

\newcommand{\CL}{\mathcal{L}}

\newcommand{\Ch}{1}

\newcommand{\HH}{H}

\newcommand{\defeq}{\coloneqq}

\renewcommand{\Pr}{\BP}

\newtheorem{theorem}{Theorem}[section]
\newtheorem{claim}[theorem]{Claim}

\newtheorem{corollary}[theorem]{Corollary}
\newtheorem{lemma}[theorem]{Lemma}

\numberwithin{equation}{section}
\numberwithin{figure}{section}
\theoremstyle{definition}
\newtheorem{definition}{Definition}[section]

\newcommand{\expref}[2]{{\texorpdfstring{\hyperref[#2]{#1~\ref{#2}}}{#1~\ref{#2}}}} 
\newcommand{\secref}[1]{\expref{Section}{#1}}
\newcommand{\thmref}[1]{\expref{Theorem}{#1}}
\newcommand{\defref}[1]{\expref{Definition}{#1}}

\newcommand{\clmref}[1]{\expref{Claim}{#1}}
\newcommand{\propref}[1]{\expref{Proposition}{#1}}

\newcommand{\appref}[1]{\expref{Appendix}{#1}}
\newcommand{\lref}[1]{\expref{Lemma}{#1}}
\newcommand{\corref}[1]{\expref{Corollary}{#1}}

\newcommand{\pref}[1]{\expref{Proposition}{#1}}

\begin{document}

\title{Majorizing Measures for the Optimizer}

\author{Sander Borst$^1$\thanks{\texttt{sander.borst@cwi.nl}. Supported by the ERC Starting grant QIP--805241.} \and
Daniel Dadush$^1$\thanks{\texttt{d.n.dadush@cwi.nl}.
Supported by the ERC Starting grant QIP--805241.
} \and
Neil Olver$^2$\thanks{\texttt{n.olver@lse.ac.uk}. Supported by the NWO VIDI grant 016.Vidi.189.087.
} \and
Makrand Sinha$^1$\thanks{ \texttt{makrand.sinha@cwi.nl}. Supported by the NWO VICI grant 639.023.812.}
}

\date{%
	\footnotesize
    $^1$Centrum Wiskunde and Informatica, Amsterdam, The Netherlands\\%
    $^2$London School of Economics and Political Science, London, UK
}

\maketitle
\begin{abstract}
The theory of majorizing measures, extensively developed by Fernique, Talagrand
and many others, provides one of the most general frameworks for controlling the
behavior of stochastic processes. In particular, it can be applied to derive
quantitative bounds on the expected suprema and the degree of continuity of
sample paths for many processes.

One of the crowning achievements of the theory is Talagrand's tight alternative
characterization of the suprema of Gaussian processes in terms of majorizing
measures. The proof of this theorem was difficult, and thus considerable effort
was put into the task of developing both shorter and easier to understand
proofs. A major reason for this difficulty was considered to be theory of
majorizing measures itself, which had the reputation of being opaque and
mysterious. As a consequence, most recent treatments of the theory (including by
Talagrand himself) have eschewed the use of majorizing measures in favor of a purely
combinatorial approach (the \emph{generic chaining}) where objects based on
sequences of partitions provide roughly matching upper and lower bounds on the
desired expected supremum.
 
In this paper, we return to majorizing measures as a primary object of study,
and give a viewpoint that we think is natural and clarifying from an
optimization perspective. As our main contribution, we give an algorithmic proof
of the majorizing measures theorem based on two parts:
\begin{itemize}
\item We make the simple (but apparently new) observation that finding the best majorizing measure can be cast as a convex program.
  This also allows for efficiently computing the measure using off-the-shelf methods from convex optimization.
  \item We obtain tree-based upper and lower bound certificates by \emph{rounding}, in a series of steps, the primal and dual solutions to this convex program.
\end{itemize}
While duality has conceptually been part of the theory since its beginnings, as far as we are aware no explicit link to convex optimization has been previously made.

\end{abstract}

\renewcommand{\E}{\BE}
\renewcommand{\diam}{\mathbf{D}}
\newcommand{\CN}{\mathcal{N}}
\renewcommand{\deg}{\mathrm{deg}_+}

\section{Introduction}

Let $(Z_x)_{x \in X}$ denote a family of centered (mean zero) jointly Gaussian random
variables, indexed by points of a set $X$. A fundamental statistic of
such a process is the expected supremum $\E[\sup_{x \in X} Z_x]$, which provides
an important measure of the size of the process. This statistic has applications
in a wide variety of areas. We list some relevant examples. In convex geometry,
one can associate a process to any symmetric convex body $K$, whose supremum
gives lower bounds on the size of the largest nearly spherical sections of
$K$~\cite{Milman71}. In the context of dimensionality reduction, one can
associate a Gaussian process to any point set $S$ in $\R^d$ whose squared expected supremum
upper bounds the projection dimension needed to approximately preserve distances
between points in $S$~\cite{Gordon88,ORS18}. In the study of Markov Chains, the square of the
expected supremum of the Gaussian free field of a graph $G$ was shown to
characterize the cover time of the simple random walk on $G$~\cite{DLP12}.  

The above list of applications, which is by no means exhaustive, help motivate
the interest in many areas of Mathematics for obtaining a fine grained
understanding of such suprema. We now retrace some of the key 
developments in the theory of Gaussian processes leading up to Talagrand's
celebrated majorizing measure theorem~\cite{Talagrand87}, which gives an
alternate characterization of Gaussian suprema in terms of an optimization
problem over measures on $X$. The goal of this paper is to give a novel
optimization based perspective on this theory, as well as a new
\emph{constructive} proof of Talagrand's theorem. For this purpose, some of the
earlier concepts, in particular, majorizing measures, will be central to the
exposition. We will also cover some generalizations of the theory to the
non-Gaussian setting, as our results will be applicable there as well.
Throughout our exposition, we rely on the terminology introduced by van
Handel~\cite{vH16} for the various combinatorial objects within the theory
(i.e., labelled nets, admissible nets and packing trees).

\subsection{Bounding the Supremum of Stochastic Processes}

In what follows we use the notation $A \lesssim B$ ($A \gtrsim B$) if there
exists an absolute constant $c > 0$ such that $A \leq c B$ ($cA \geq B$). We
use $A \asymp B$ to denote $A \lesssim B$ and $A \gtrsim B$.

A first basic question one may ask is what information about the Gaussian
process $(Z_x)_{x \in X}$ is sufficient to exactly characterize the expected supremum? An
answer to this problem was given by Sudakov~\cite{Sudakov71}, strengthening a
result of Slepian~\cite{Slepian62}. Sudakov showed that it is uniquely identified by
the natural (pseudo) distance metric 
\begin{equation}
\label{eq:gauss-metric}
d(u,v) \defeq \E[(Z_{u}-Z_{v})^2]^{1/2}, ~~~~ \forall u,v \in X. 
\end{equation}
In fact, Sudakov proved the following stronger comparison theorem: if $(Y_x)_{x \in
X}$ and $(Z_x)_{x \in X}$ are Gaussian processes on the same index set $X$ and for every
$u,v \in X$, it holds that $\E[(Y_u-Y_v)^2] \leq \E[(Z_u-Z_v)^2]$, then $\E[\sup_x Y_x] \leq
\E[\sup_x Z_x]$.

Given the above, it is natural to wonder what properties of the metric space $X$
allow us to obtain upper and lower bounds on $\E[\sup_{x \in X} Z_x]$? A first
intuitively relevant quantity is the diameter of $X$ defined by $\diam(X) \defeq \sup_{u,v \in X} d(u,v)$. For any $u,v \in X$, we have
the following simple lower bound:
\begin{equation}\label{eq:diam-bnd}
\E[\sup_{x} Z_x] \geq \E[\max \{Z_u, Z_v\}] = 
\E[\max \{Z_u-Z_v,0\}] + \E[Z_v] = \tfrac12\E [| Z_u - Z_v|]
= \frac{d(u,v)}{\sqrt{2\pi}}.
\end{equation}
Here, we use that $\E[Z_v] = 0$, and that $Z_u - Z_v$ is  Gaussian with variance $d(u,v)^2$.
Thus $\E[\sup_x Z_x] \geq \diam(X)/\sqrt{2\pi}$.

Instead of looking at two maximally separated points, one
might expect to get stronger lower bounds using a large set of
well-separated points in $X$. Such an inequality was given by
Sudakov~\cite{Sudakov71}, who showed that 
\[
\max_{r > 0} r \sqrt{\log N_X(r)} \lesssim \E[\sup_x Z_x],      
\]
where $N_X(r) \defeq \min \{|S| \mid S \subseteq X, \forall x \in X, \min_{s \in S}
d(x,s) \leq r\}$ is the minimum size of an $r$-net of $X$. This is in fact a
direct consequence of Sudakov's comparison theorem. Precisely, the restriction of the
process $\{Z_x\}_{x \in X}$ to a suitable $r$-net $S$, chosen greedily so that every two points in
$S$ are at distance at least $r$, majorizes the maximum of $|S| \geq N_X(r)$
independent Gaussians with standard deviation $r/\sqrt{2}$, where a standard
computation then yields the left-hand side. 

On the upper bound side, Dudley~\cite{Dudley67} proved that the covering
numbers can in fact be \emph{chained} together to upper bound the supremum: 
\begin{equation}
\label{eq:dudley}
\E[\sup_x Z_x] \lesssim \int_0^\infty \sqrt{\log N_X(\eps)} d\eps.
\end{equation}
Note that the integral can be restricted to the range $\eps \in (0,\diam(X)]$, since
$\log N_X(\diam(X)) = \log 1 = 0$. Dudley's proof of this inequality was extremely
influential and showed the power of combining simple tail bounds on pairs of
variables $Z_u-Z_v$ to get a global bound on the supremum. In particular, the
main inequality used in Dudley's proof is the standard Gaussian tail bound: for
$u,v \in X$, and for any $s > 0$  
\begin{equation}
\label{eq:gauss-tail}
\Pr[|Z_u-Z_v| \geq d(u,v) \cdot s] \leq 2 e^{-s^2/2}.
\end{equation}
The strategy of combining the above inequalities to control the maximum of a
process is what is now called chaining.

\paragraph{\bf Basics of Chaining.} The concept of chaining is central to this
paper, so we explain the basic mechanics here. As it will be more convenient for
the exposition, we will more directly work with symmetric version of the supremum
\[
\sup_{x_1,x_2 \in X} Z_{x_1}-Z_{x_2}
\]
which is always non-negative. Note that since $(Z_x)_{x \in X}$ and $(-Z_x)_{x
\in X}$ are identically distributed, 
\[
\E\left[\sup_{x_1,x_2 \in X} Z_{x_1}-Z_{x_2}\right] = \E\left[\sup_{x \in X} Z_x\right] + \E\left[\sup_{x
\in X} -Z_x\right] = 2\E\left[\sup_{x \in X} Z_x\right], 
\]
and thus the expected supremum is the same after dividing by $2$.

From here, instead of bounding the expectation, we focus on upper bounding the
median of $\sup_{x_1,x_2 \in X} Z_{x_1}-Z_{x_2}$, which is known to be within a
constant factor of the expectation. Precisely, we seek to compute a number $M >
0$ such that $\Pr[\sup_{x_1,x_2 \in X} Z_{x_1}-Z_{x_2} \geq M] \leq 1/2$. To
arrive at such bounds, we define the notion of a \emph{chaining tree}.

\begin{definition}[Chaining Tree]
\label{def:chaining-tree}
A (Gaussian) chaining tree $\CCT$ for a finite metric space $(X,d)$ is a rooted
spanning tree on $X$, with root node $w \in X$, together with probability labels
$p_e \in (0,1/2)$, for each edge $e \in E[\CCT]$. The edge probabilities are
required to satisfy $\sum_{e \in E[\CCT]} p_e \leq 1/2$. For each edge $\{u,v\} =
e \in E[\CCT]$, we define the induced edge length $l_e \defeq l_e(p_e,e)$ to satisfy 
\begin{equation}
\label{eq:gauss-edge}
\Pr_{{Z \in \CN(0,d(u,v)^2)}} [|Z| \geq l_e] = p_e.
\end{equation}
For each $x \in X$, let $\CP_x$ denote the unique path from $x$ to
the root $w$ in $\CCT$. We define the value of $\CCT$ to be
\begin{equation}
\label{eq:chain-value}
\val(\CCT) = \max_{x \in X} \sum_{e \in \CP_x} l_e.
\end{equation}
\end{definition}

For a Gaussian process $(Z_x)_{x \in X}$, where $d$ is the induced metric as
in~\eqref{eq:gauss-metric}, for any chaining tree $\CCT$ on $X$, we now show that 
\begin{equation}
\label{eq:chain-bound}
\Pr\left[\sup_{x_1,x_2} Z_{x_1}-Z_{x_2} \geq 2\cdot \val(\CCT)\right] \leq 1/2.
\end{equation}

By construction, for any edge $\{u,v\} \in E[\CCT]$ we first note that
\[
\Pr[|Z_u-Z_v| \geq l_e] = p_e,
\] 
recalling that $Z_u-Z_v$ is distributed as $\CN(0,d(u,v)^2)$. Since $\sum_{e \in
\CCT} p_e \leq 1/2$, by the union bound the event $\CE$ defined as 
``$|Z_u-Z_v| \leq l_{u,v}$, $\forall \{u,v\} \in E[\CCT]$'', holds with probability
at least $1/2$. For $x \in X$, let us now define $\CP_x$ to be the unique path
from the root $w$ to $x$ in $\CCT$.

Conditioning on the event $\CE$, by the triangle inequality
\begin{equation}
\label{eq:tree-path}
|Z_x-Z_w| \leq \sum_{\{u,v\} \in \CP_x} |Z_u-Z_v| \leq \sum_{e \in \CP_x} l_e.
\end{equation}
Applying the triangle inequality again, we have that
\[
\sup_{x_1,x_2} Z_{x_1}-Z_{x_2} \leq 2 \sup_{x \in X} |Z_x-Z_w| \leq 2 \cdot \val(\CCT).
\]
The bound~\eqref{eq:chain-bound} now follows since the above occurs with
probability at least $1/2$. 

To work with such chaining trees, it is important to have easy approximations of
the edge lengths used above. For $e = \{u,v\} \in E[\CCT]$ and $p_e \in (0,1/2]$,
it is well known that
\begin{equation}
\label{eq:gauss-edge-apx}
l_e \asymp d(u,v)\sqrt{\log(1/p_e)}.
\end{equation}

Note that by the standard Gaussian tail bounds \eqref{eq:gauss-tail}, for any $p
\in (0,1)$, we have the upper bound, $l_e \leq d(u,v)\sqrt{2\log(2/p_e)}$.

To relate to earlier lower bounds, it is instructive to see that $\val(\CCT)
\gtrsim \diam(X)$, for any chaining tree. Firstly, since each $p_{u,v} \in (0,1/2]$,
by~\eqref{eq:gauss-edge-apx}, it follows that $l_{u,v} \gtrsim d(u,v)$. From here, for any pair of
points $u,v \in X$, by the triangle inequality $2\cdot \val(\CCT)$ pays for the
cost of going from $u$ to the root $w$ and from $w$ to $v$, yielding the desired
upper bound on the diameter.

\paragraph{\bf Chaining beyond Gaussians.} Importantly, in the above framework,
the only element specific to Gaussian processes is the edge length
function~\eqref{eq:gauss-edge}. As our results will apply to this more general setting, we
explain how chaining can straightforwardly be adapted to work with processes
satisfying appropriate tail bounds. 

Let us examine a jointly distributed sequence of random variables $(Z_x)_{x \in
X}$ indexed by a metric space $(X,d)$. To constrain the process we will
make the following assumptions on the tails. Let $f: \R_+ \rightarrow \R_+$ be a
continuous and non-increasing probability density function on the non-negative
reals and let $F(s) = \int_s^\infty f(t)dt$ denote the complementary cumulative distribution
function. Then, for all $x_1,x_2 \in X$ and $s \geq 0$, we assume that      
\begin{equation}
\label{eq:tail-assumption}
\Pr[|Z_{x_1}-Z_{x_2}| \geq d(x_1,x_2) \cdot s] \leq F(s).
\end{equation}

\begin{definition}[Chaining Functional]
\label{def:chain-func}
We define the \emph{chaining functional} induced by $f$ to be $h(p) \defeq
h_f(p) = F^{-1}(p)$, for $p \in (0,1]$, which is well-defined since $f$ is
non-increasing. Note that $h(1) = 0$ and that $h(p)$ is strictly decreasing on
$(0,1]$. We say that $h$ is of \emph{log-concave type} if the density $f$ is
log-concave.
\end{definition}

A property that we make crucial use of is that $h(p)$ is, in fact, a convex
function of $p \in (0,1]$. To see this, for $p \in (0,1)$, a direct computation
yields that $h'(p) = 1/F'(h(p)) = -1/f(h(p))$, where the derivative is
well-defined since $f$ is continuous and non-decreasing. Since $f(s)$ is
non-increasing and $h(p)$ is strictly decreasing, $h'(p)$ is non-decreasing and
hence, $h$ is convex. Throughout the rest of the paper, we will mainly be
interested in chaining functionals of log-concave type.

To apply the chaining framework to the process $(Z_x)_{x \in X}$, we use a
chaining tree $\CCT$ exactly as in \defref{def:chaining-tree} except that we now
compute the edge lengths according to the chaining functional $h$. Specifically, for
$e = \{u,v\} \in E[\CCT]$ and probability $p_e \in (0,1)$, we define 
\begin{equation}
\label{eq:gen-edge}
l_e \defeq l_e(e,p_e) \defeq d(u,v) \cdot h(p_e).
\end{equation}
We now define $\val_h(\CCT)$ exactly as in~\eqref{eq:chain-value}, using $h$ to
compute the edge lengths. 

With this setup, with an identical proof to the previous section, we have
the inequality
\[
\Pr\left[\sup_{x_1,x_2 \in X} Z_{x_1}-Z_{x_2} \geq 2\cdot \val_{h}(\CCT)\right] \leq \tfrac12.
\]

As in the Gaussian setup, it is useful to keep in mind what the ``trivial''
diameter lower bound on $\valch(\CCT)$ should be. Since the edge probability
$p_e \in (0,1/2]$, for $e = \{u,v\} \in E[\CCT]$, we have that $l_e \geq
d(u,v) \cdot h(1/2)$. Therefore, for any $u,v \in X$, by the triangle inequality, the
cost of the paths from $u$ or $v$ to the root $w$ is at least $d(u,v) \cdot h(1/2)$ for any chaining tree $\CT$. In
particular, for any chaining tree $\CCT$, we derive the lower bound
\begin{equation}
\label{eq:trivial-lb}
2\cdot \valch(\CCT) \geq \diam(X) \cdot h(1/2).  
\end{equation}

It is important to note that the Gaussian chaining setup is indeed a special
case of the above. Precisely, in that setup the edge lengths are $l_{u,v} \defeq
d(u,v)\cdot h_f(p_{u,v})$, where $f(s) = \sqrt{\frac{2}{\pi}}
e^{-s^2/2}$ is the density of the absolute value of the standard Gaussian.

\paragraph{\bf Dudley's Construction.} To gain intuition about how to apply
the chaining framework, we now explain how to build and analyze the chaining tree
used in Dudley's inequality. For simplicity of notation, let us assume that the
diameter $\diam(X)=1$. For each $k \geq 0$, let $\mathcal{N}_k$ denote a
$2^{-k}$-net of $X$ of minimum size, i.e., satisfying $|\mathcal{N}_k| =
N_X(2^{-k})$. By our diameter assumption, $\mathcal{N}_0 = \{w\}$ is clearly a
single point, which gives the root of the tree $\CCT$. From here, we
construct the tree by induction on $k \geq 1$. At iteration $k \geq 1$, we
attach each element of $\CN_k$ not already in $\CCT$ to a closest point in $\CCT$.
From here, we set the edge probability $p_{u,v} = p_k \defeq
2^{-(k+1)}/|\mathcal{N}_k|$ and let $l_{u,v} > 0$ be minimal subject to
$\Pr[|Z_u-Z_v| \geq l_{u,v}] \leq p_k$. This completes the construction.

To analyze the tree $\CCT$, we make the following observations. Firstly,
the number of edges we add to the tree at iteration $k$ is at most
$|\mathcal{N}_k|$. Therefore, the total probability sum is at most
$\sum_{k=1}^\infty |\mathcal{N}_k|\cdot p_k = 1/2$, and hence $\CCT$ is a valid
chaining tree. Second, any edge $\{u,v\}$ added during iteration $k$ satisfies
$d(u,v) \leq 2^{-k+1}$. Consequently, by the Gaussian tail bound~\eqref{eq:gauss-tail}, 
\[
l_{u,v} \lesssim d(u,v)\sqrt{\log(1/p_k)} \lesssim 2^{-k+1}\left(\sqrt{\log
N_X(2^{-k})} + \sqrt{k+1}\right). 
\]
In particular, the value of $\CCT$ satisfies
\[
\val(\CCT) \lesssim \sum_{k=1}^\infty 2^{-k+1}\left(\sqrt{\log
N_X(2^{-k})}+\sqrt{k+1}\right) \lesssim 1 + \sum_{k=1}^\infty 2^{-k+1}
\sqrt{\log(N_X(2^{-k}))}.
\]
One can now easily show that the above expression is upper bounded
by~\eqref{eq:dudley} by discretizing the range of the integral along powers of
$2$ (recalling that $\diam(X)=1$). 

\paragraph{\bf The Method of Majorizing Measures.} Given the above, it is natural
to wonder how one might construct an \emph{optimal} chaining tree for a given
process $(Z_x)_{x \in X}$. A principal goal of this paper will be to
give efficient constructions for such trees. At first sight, this may seem like a daunting task,
as one must somehow simultaneously optimize over all spanning trees and edge
probabilities. Nevertheless, a major step towards this goal was achieved for Gaussian
processes by Fernique~\cite{Fernique75}, who proved the following remarkable theorem:   
\begin{equation}
\label{eq:maj-meas-ub}
\E\left[\sup_{x \in X} Z_x\right] \lesssim \primal_2(X) \defeq \inf_{\mu} \sup_{x \in X} \int_0^\infty
g(\mu(B(x,\eps))) d\eps.
\end{equation}
Some definitions are in order. Firstly, the infimum over $\mu$ is taken over all
probability measures on $X$. Secondly, $g(p) \defeq \sqrt{\log(1/p)}$, for $p \in
[0,1]$ corresponds to (an approximation of) the Gaussian edge length function
in~\eqref{eq:gauss-edge}. Lastly, $B(x,\eps) = \{y \in X: d(x,y) \leq r\}$ is the
metric ball of radius $r$ around $x$, where $d$ is the canonical metric induced
by the Gaussian process.

Importantly, the natural analogue of $\primal_2$ for the general setup
in~\eqref{eq:tail-assumption} also yields upper bounds on the expected supremum,
provided the tails of $f$ decay sufficiently quickly. In particular, for
$(Z_x)_{x \in X}$ satisfying~\eqref{eq:tail-assumption}, for any ``nice enough''
$f$, we have that 
\begin{equation}
\label{eq:gen-maj-meas-ub}
\E\left[\sup_{x_1,x_2 \in X} Z_{x_1}-Z_{x_2}\right] \lesssim \primal_h(X) \defeq \inf_{\mu} \sup_{x \in X} \int_0^\infty
h(\mu(B(x,\eps))) d\eps,
\end{equation}
where $h$ is as in~\eqref{eq:gen-edge}. Note that since $h(1) = 0$, one can
truncate the range of the integral to $r \in [0,\diam(X)]$. Very general results
of the above type can be found in~\cite{Talagrand90,Bednorz06}. We note that the
requirements of the process in these works are parametrized is a somewhat
different way in terms of Orlicz norms. In this work, we will focus on the
setting where the chaining functional $h$ is of log-concave type (where the tail
density $f$ is log-concave), where these different parametrizations are
equivalent. Prototypical examples in this class are the tail densities of
\emph{exponential type}, which are proportional to $e^{-x^q}$, $x \geq 0$, for
$q \geq 1$, and where $h(p) \asymp \ln^{1/q}(1/p)$ for $p \in (0,1/2)$.

Given that any probability measure $\mu$ can be used to upper bound the expected
supremum, Fernique dubbed the above technique the method of \emph{majorizing
measures}. It is worthwhile to note that Fernique did not prove
inequality~\eqref{eq:maj-meas-ub} via chaining. He relied instead on a more
general technique, which first proves a generic concentration inequality for
real valued functions on the metric space, and recovers the desired inequality
by averaging over the ensemble of functions induced by the process. The
fact that one can recover the same bound via chaining for Gaussian processes
would only be proved later, at first implicitly in Talagrand~\cite{Talagrand87},
and explicitly in~\cite{T01}, where the latter work also covered processes of
exponential type mentioned above. 

As noted above, the quantity $\primal_2(X)$ and more generally $\primal_h(X)$
(for $h$ of log-concave type), rather miraculously models the value of the best chaining
tree as a continuous optimization problem. 
As majorizing measures may seem like rather
opaque objects at first sight, we believe it is instructive to note
that from a chaining tree $\CCT$, one can construct a measure $\mu$ whose
value in \eqref{eq:maj-meas-ub} is at most $3\cdot \val_h(\CCT)$.
The construction is simple: set $\mu_w = 1/2$ on the root $w$, and for each $e=\{u,v\} \in E[\CCT]$ (with $v$ closer to the root than $u$),
set $\mu_u = p_e$.
The details of the comparison can be found in \lref{lem:cctbound} in the appendix.

From the above discussion, we see that the majorizing measures are indeed powerful tools
for upper bounding suprema. Given this, together with the many tools for 
\emph{lower bounding} Gaussian suprema (which are not available in general),
Fernique~\cite{Fernique75} conjectured that majorizing measures should fully
characterize the expected supremum of Gaussian processes. This conjecture was
verified in the ground-breaking work of Talagrand~\cite{Talagrand87}, which is
now called the majorizing measures theorem: 

\begin{theorem}[Fernique-Talagrand \cite{Fernique75,Talagrand87}] \label{thm:fernique-talagrand}
For any centered Gaussian process $(Z_x)_{x \in X}$ over the metric space $X=(X,d)$,
where $d$ is the canonical metric induced by the process, we have 
\[ 
    \BE\Bigl[\,\sup_{x \in X} Z_x\,\Bigr] \asymp \primal_2(X).
\]
\end{theorem}

The original proof of the majorizing measure theorem~\cite{Talagrand87} was
considered notoriously difficult. Due to its importance in the theory of
stochastic processes, many simpler as well as different proofs were
found~\cite{Talagrand92,T96,T01,MN13,B12,vH16}, often by Talagrand himself. 

As stated at the beginning of the introduction, the goal of this paper is to
give an alternative constructive proof of this theorem using a convex
optimization approach. In particular, our starting point is the simple
observation that $\primal_h(X)$ is in fact a convex program, which follows
directly from the convexity of $h$\footnote{The formulation $\primal_2(X)$ is
``essentially convex''. This is because $g(p)$ is only convex on the interval
$[0,1/\sqrt{e}]$, which is easily remedied. We note this non-convexity is
principally due to $g(p)$ being a poor approximation of~\eqref{eq:gauss-edge}
for $p \in [1/\sqrt{e},1]$.}. While simple (and most certainly known to
experts), we have not seen this observation leveraged in earlier proofs. In our
context, convexity will allow for near-optimal solutions to $\primal_h(X)$ to be
efficiently computed using off-the-shelf methods. Furthermore, convex duality
will allow us to inspect the structure of solutions to natural dual program(s)
for $\primal_h(X)$, enabling us to reason about lower bounds. Our proof will
operate entirely at the level of the metric space, and will produce a natural
combinatorial variant of an optimal primal-dual solution pair for $\primal_h(X)$,
namely a chaining tree and packing tree (defined shortly). These solutions
will in fact be obtained by ``rounding'' solutions to the corresponding
continuous programs. Specializing to the Gaussian case, we recover the
majorizing measure theorem by an easy comparison between the value of the
Gaussian supremum and the value of the combinatorial solutions (which are tailor
made for this purpose). This strategy has the benefit of clearly separating the
role of the metric space and the role of the Gaussian process, which are often
intertwined in difficult to disentangle ways in many proofs.

We now review some of the key ideas in the known proofs of the Majorizing measures theorem, which will be important for our approach as well. In particular, we will require appropriate dual analogues
to chaining trees. 

\paragraph{\bf Primal Proof Strategies.}  
Given what we have seen so far, a main missing ingredient is a stronger form of
lower bound for the value of the Gaussian supremum (noting that chaining already
provides the upper bound). For this purpose, we examine the natural functional
induced by the process on subsets of $X$, defining  
\begin{equation}
\label{eq:gauss-func}
G(S) \defeq \E\left[\sup_{x \in S} Z_x\right], ~~\forall S \subseteq X. 
\end{equation}
The following functional inequality, named the ``super-Sudakov'' inequality
in~\cite{vH16}, was proven in~\cite{Talagrand92}: there exists $\gamma \in
(0,1)$, such that given an $r$-separated (non-empty) subsets $A_1,\dots,A_N
\subseteq S$, i.e., satisfying $d(A_i,A_j) \geq r$, $\forall i \neq j$,  and $\diam(A_i) \leq \gamma r$, $\forall i \in [N]$, then
\begin{equation}
\label{eq:super-sudakov}
G(S) \geq \gamma \cdot  r \cdot g(1/N) + \min_{i \in [N]} G(A_i),
\end{equation}
where $g(x) = \sqrt{\log 1/x}$ for $x \in [0,1]$, as before.

In~\cite{Talagrand92}, Talagrand gave a construction which takes a functional
$G$ on $X$ satisfying~\eqref{eq:super-sudakov}, and produces (a variant of) a
chaining tree $\CCT$ satisfying $\val(\CCT) \lesssim G(X)$. Talagrand's
construction is based on a recursive partitioning scheme, where the partitions
roughly correspond to subtrees, which greedily chooses metric balls of large $G$
value to construct the partition. We note that this construction comes in
different flavors, each yielding more structured versions of chaining trees
(i.e.~labelled nets~\cite{Talagrand92} and admissible nets~\cite{T96}). By
instantiating $G$ to be the functional given by~\eqref{eq:gauss-func} immediately yields
\thmref{thm:fernique-talagrand}. While Talagrand's construction was
certainly algorithmic, the Gaussian functional $G$ is not easy to compute (at
least deterministically). As mentioned previously, in~\cite{T01}, Talagrand also
gave another procedure that directly converts measures to chaining trees. Note
that this yields a good chaining tree from a good measure, but by itself does
not yield~\thmref{thm:fernique-talagrand}. 

\paragraph{\bf Dual Proof Strategies.}

One reason the ``difficult'' Gaussian functional $G$ was required to
prove \thmref{thm:fernique-talagrand} is that there was no simple dual object to
compare to that certifies a lower bound. From the convex optimization
perspective, this should morally correspond to a solution to the dual of
$\primal_2(X)$ (or $\primal_h(X)$). Such an object, called a \emph{packing
tree} in the terminology of~\cite{vH16}, was in fact developed in Talagrand's
original proof~\cite{Talagrand87} for the Gaussian case, and extended to general
chaining functionals in~\cite{LT91}.  

\begin{definition}[Packing Tree] 
\label{def:separated-tree}
Let $\alpha \in (0,\frac{1}{10}]$. An $\alpha$-packing tree $\CT$ on a finite metric space
$(X,d)$ is a rooted tree on subsets of $X$, with root node $W \subseteq X$,
together with a labelling $\chi: \CT \rightarrow \Z_{\ge 0}$. We enforce that every leaf
node $V \in \CT$ is a singleton, i.e., $V=\{x\}$ for some $x \in X$. We denote
${\rm leaf}(\CT) \subseteq X$ to be the union of all leaf nodes of $\CT$. Every
node $V \in \CT$ has a (possibly empty) set of children $C_1,\dots,C_k \subseteq V$
which are pairwise disjoint. We let $\deg(V) \defeq k$ denote the \emph{number of children}
of $V$. We enforce the follow metric properties on $\CT$: 
\begin{enumerate}
\item \label{eq:decrease} For any child $C$ of $V \in \CT$, we have that $\diam(C) \le \alpha^{\chi(V) + 1} \cdot\diam(X)$. 
\item \label{eq:separation} For $V \in \CT$ and distinct children
$C_1,C_2$ of $V$, we have $d(C_1,C_2) \ge \frac{1}{10}\alpha^{\chi(V)} \cdot \diam(X)$. 
\end{enumerate}

The value of an $\alpha$-packing tree $\CT$ with respect to a chaining functional
$h$ is defined as 
\begin{equation}
\label{eq:sep-tree-value}
\valsep(\CT) \defeq ~\inf_{x \in {\rm leaf}(\CT)} \sum_{V
\in \CP_x \setminus \{x\}}
\alpha^{\chi(V)} \cdot \diam(X) \cdot h(1/\deg(V)),
\end{equation}
where $\CP_x$ is the unique path from the root $W$ to the leaf $\{x\}$. We use
the shorthand $\val_2(\CT)$ to denote the value with respect to the Gaussian
functional $g$.  
\end{definition}

We remark that we do not count the edge going to the parent in $\deg(V)$ mostly for notational convenience --- in this case, nodes with a sole child do not contribute to the value of the packing tree.  Also, there is quite a bit of flexibility in the parameters of the
packing tree, which are chosen above for convenience. Packing trees are objects
that allow us to chain lower bounds together in analogy to upper bounds via
chaining trees. The combinatorial structure of a packing tree is more
constrained than that of a chaining tree however, and their construction (at
least more from the perspective of the analysis) is more delicate.

In the Gaussian setting, an $\alpha$-packing tree $\CT$ is perfectly tailored
for combining the  ``super-Sudakov" inequalities given by~\eqref{eq:super-sudakov}.
In particular, for $\alpha = 1/(2\gamma)$, a  direct proof by induction starting from the leaves of the tree certifies that $G(X) \gtrsim \val_2(\CT)$ (see Theorem 6.36 in \cite{vH16}). This was in
fact first established in~\cite{Talagrand87} using Slepian's lemma instead
of~\eqref{eq:super-sudakov}. Independently of any process however, they also
directly serve as combinatorial lower bounds for $\primal_h(X)$.

\begin{restatable}{lemma}{weakdual}
\label{lem:weak-comb-dual}
Let $\alpha \in (0, \frac{1}{10}]$. For a finite metric space $(X,d)$, an $\alpha$-packing tree $\CT$ on $X$, and any chaining functional $h$, we have 
\[ \primal_h(X) \geq \frac{1}{2}(1-\alpha) \cdot \valsep(\CT).  \]
\end{restatable} 
While known to experts, it is not so easy to find combinatorial proofs of the
above inequality, i.e.~not related to a process, in the literature (see for
example Exercise 6.12 in~\cite{vH16} or Lemma 3.7 in~\cite{DLP12}). We provide a proof in the appendix.

Talagrand's original proof of the majorizing measures theorem worked almost entirely on
the dual side. As generalized in~\cite{LT91}, the main work in the proof was in
fact to construct an $\alpha$-packing tree $\CT$ satisfying $\val_h(\CT) \gtrsim
\primal_h(\CT)$ (for $h$ of log-concave type). As for the primal side, the
construction is based on similar greedy ball (sub-)partitioning using an appropriate
functional $H$ on $X$ satisfying a so-called ``super-chaining'' inequality in
the terminology of~\cite{vH16}. Specifically, for any set $S \subseteq X$, and
a partition $S = \sqcup_{i=1}^N P_i$, $H$ satisfies
\begin{equation}
\label{eq:super-chaining}
H(S) \leq \max_{i \in [N]} \beta \cdot \diam(S)h(1/(i+1)) + H(P_i),
\end{equation}   
for some absolute constant $\beta > 0$. Interestingly, the functional $H$
used in~\cite{Talagrand87,LT91} was a variant of $\primal_h(X)$, which is
deterministically computable, and not the Gaussian functional in the case $h=g$
(though this works as well~\cite{vH16}). This construction was in fact leveraged
in~\cite{DLP12} to give a deterministic polynomial time dynamic programming
algorithm for computing a nearly optimal packing tree. 

\subsection{Our Results}

\subsubsection{A Constructive Min-Max Theorem}

The main result of this paper is the following constructive variant of the
combinatorial core of the majorizing measures theorem.

\begin{theorem} 
\label{thm:constructive}
Let $(X,d)$ be an $n$ point metric space, $h$ be a chaining functional of
log-concave type. Then there is a deterministic algorithm which computes a
chaining tree $\CCT^*$ and a $\frac1{10}$-packing tree $\CT^*$ satisfying
\[
\val_h(\CCT^*) \asymp \val_h(\CT^*),
\]
using $\tilde{O}(n^{\omega+1})$ arithmetic operations and evaluations of $h$ and
$h'$, where $\omega \leq 2.373$ is the matrix multiplication constant. 
\end{theorem}

We note that the packing tree parameter $\frac1{10}$ can be made smaller at the cost
increasing the hidden constant in the $\asymp$ notation. Recall that for any
pair of trees $\CCT$ and $\CT$ as above, we have already seen that
\begin{equation}
\label{eq:gen-weak-duality}
\val_h(\CCT) \gtrsim \primal_h(X) \gtrsim \val_h(\CT),
\end{equation}
so the pair in \thmref{thm:constructive} form a nearly-optimal primal-dual
pair. Furthermore, in the Gaussian setting where $h=g$, replacing $\primal_2(X)$
above by $\E[\sup_x Z_x]$ corresponds to the ``easy direction'' of the
majorizing measures theorem. Plugging in the solutions from
\thmref{thm:constructive} immediately yield the hard direction of the
theorem. This allows us to view the metric space part of the majorizing measure
theorem as an instance of a combinatorial min-max theorem.  We remark that such a combinatorial min-max characterization of the Majorizing Measures theorem was already observed by Gu{\'e}don 
and Zvavitch~\cite{GZ03}. They showed that the value of the optimal packing tree defines a functional that satisfies the super-Sudakov inequality; when combined with Talagrand's framework, this  implies the combinatorial min-max theorem described above. Although this does not directly yield deterministic constructions of nearly-optimal chaining or packing trees.

\cite{DLP12} essentially used this observation of \cite{GZ03} along with a suitable dynamic program to give an efficient deterministic algorithm to compute nearly optimal packing trees, as mentioned previously. For nearly optimal chaining trees, we make the simple observation (which seems to have gone unnoticed) that
these can extracted from Talagrand's~\cite{T01} ``rounding'' algorithm applied
to a nearly optimal solution for the efficiently solvable convex program
$\primal_h(X)$. By themselves however, these algorithms do not directly say much
about how the values of these different trees relate to each other.    

In \thmref{thm:constructive}, we build further on the convex programming
approach. At a high level, we build the primal and dual solutions at the same
time and rely on convex programming duality to ensure they have (nearly) the
same value. In essence, we replace the ``magic functionals''  satisfying super-Sudakov or super-chaining inequalities that appear in Talagrand's
constructions with convex duality. As
we will see in the next section, the dual objects will also correspond to
probability measures. In contrast to the primal however, where the rounding to
a suitable chaining tree can be done in one shot, the dual measures will require
multiple levels of rounding.

The primal and dual solutions we require correspond to nearly optimal primal and
dual measures associated with a saddle-point formulation of $\primal_h(X)$
(see~\eqref{eq:saddle-point} in the next subsection). There are in fact many
existing solvers that are able to compute nearly optimal solutions to such saddle
point problems, where we will rely on a recent fast solver of~\cite{JLSW20}.
This computation in fact forms the bulk of the running time of the algorithm in
\thmref{thm:constructive}. The details of this part of the algorithm are
covered in~\appref{sec:alg}. An interesting open problem is whether one can
reduce the running time of \thmref{thm:constructive} to $\widetilde{O}(n^2)$,
which would be nearly-linear in the input size (recall that an $n$ point metric
consists of $n^2$ distances). The main bottleneck is the use of an all purpose
blackbox solver~\cite{JLSW20} to approximately solve~\eqref{eq:saddle-point},
and it seems likely that an appropriately tailored first-order method could
bring the running time down to $\tilde{O}(n^2)$. 

While our main contribution is conceptual, we expect and hope that novel and
interesting applications of an ``algorithmic'' theory of chaining will be found.
As a contribution on this front, in \secref{sec:gordon}, we give an
application of \thmref{thm:constructive} in the context of derandomization:
we give a deterministic algorithm for computing Johnson-Lindenstrauss
projections achieving the guarantees of Gordon's theorem~\cite{Gordon88}, where
we rely on a chaining based proof from~\cite{ORS18}. As far as we are aware,
no prior deterministic construction was known.

\subsubsection{Simplifying the Dual of $\mathbf{\gamma}_h(X)$} \label{sec:simplifyingdual}

For simplicity of notation, throughout this section (and most of the paper), we
will assume that $(X,d)$ is a fixed $n$-point metric space and that $h$ is a
chaining functional of log-concave type satisfying $|h'(1)|=1$ (interpreted as
the left directional derivative). Under this normalization on $h$, the trivial
diameter lower bound on $\primal_h(X)$ will be at least $\diam(X)/4$, which we
will use to convert additive errors to multiplicative ones. This
normalization is without loss of generality, and can be achieved by
appropriately scaling the $h$ and the metric $d$ so that $\primal_h(X)$
remains unchanged (see \secref{sec:derivative} for a full explanation).

We now describe the dual formulation of $\primal_h(X)$ and describe the process
of simplifying it. For this purpose, we start with the basic saddle-point
formulation of $\primal_h(X)$: 
\begin{equation}
\label{eq:saddle-point}
\primal_h(X) = \min_{\mu} \max_{x \in X} \int_0^\infty h(\mu(B(x,\eps)))d\eps
= \min_{\mu} \max_{\nu} \int_X \int_0^\infty h(\mu(B(x,\eps))) d\eps d\nu(x)
\end{equation}
where $\nu$ also ranges over all probability measures on $X$ (the optimal $\nu$
above puts mass $1$ on any maximizer of $\int_0^\infty h(\mu(B(x,\eps))d\eps$).  

To obtain the dual program to $\primal_h(X)$, we interchange $\mu$ and $\nu$: 
\begin{equation}
\label{eq:real-dual}
\primal_h(X) \geq \max_{\nu} \min_{\mu} \int_X \int_0^\infty h(\mu(B(x,\eps)))d\eps
d\nu(x) \defeq \realdual_h(X).
\end{equation}
In particular, for any fixed dual measure $\nu$, we have 
\begin{equation}
\label{eq:real-dual-value}
\primal_h(X) \geq \min_{\mu} \int_X \int_0^\infty h(\mu(B(x,\eps)))d\eps d\nu(x).
\end{equation}
Since the objective $\int_X \int_0^\infty h(\mu(B(x,\eps)))d\eps \nu(x)$ is convex in
$\mu$ and linear in $\nu$, and the probability simplex is compact and convex, by
Sion's theorem the value of both convex programs is equal. That is, $\primal_h(X)
= \realdual_h(X)$. The measures required within the construction in
\thmref{thm:constructive} will be nearly optimal primal and dual measures
$\mu^*$ and $\nu^*$ to $\primal_h(X)$ and $\realdual_h(X)$ respectively. 

Unfortunately, it is not clear at this point that the dual is terribly useful.
In particular, even evaluating the objective in~\eqref{eq:real-dual-value} for a
given dual measure $\nu$ requires solving a non-trivial convex optimization
problem (note that the corresponding objective of $\primal_h(X)$ can be computed
by simply evaluating $n$ integrals). Rather surprisingly, it turns out that
for $h$ of log-concave type, one can in fact ``guess'' a near-optimal $\mu$
in~\eqref{eq:real-dual-value}, namely, we can set $\mu=\nu$. 

\begin{lemma} 
\label{lem:guess-mu}
For any probability measure $\nu$ on $X$,  we have that
\[
\int_X \int_0^\infty h(\nu(B(x,\eps))d\eps d\nu(x) \leq 2\min_{\mu} \int_X
\int_0^\infty h(\mu(B(x,\eps)))d\eps d\nu(x) + \diam(X)/e,
\]
where the minimum is taken over all probability measures $\mu$. 
\end{lemma}

The proof of the above proceeds on a ``per scale'' basis.  More precisely, for a
given $\eps
> 0$, we show that $\int_X h(\nu(B(x,2\eps)) d\nu(x) \le \int_X
h(\mu(B(x,\eps))) d\nu(x) + \Ch/e$.  This statement is easily restated in
graph-theoretic terms, by defining a graph $G=(X_1 \cup X_2,E)$, with $X_1, X_2$
both being copies of $X$ and where $x_1 \in X_1$ is adjacent to $x_2 \in X_2$ if
$d(x_1, x_2) \leq \eps$.  Then $\mu(B(x,\eps))$ corresponds to the mass under
$\mu$ within the neighborhood of $x$, and $\nu(B(x,2\eps))$ to the mass under
$\nu$ within the two-hop neighborhood of $x$. In this setting, we use a tool
from combinatorial optimization, namely, a generalization of Hall's marriage
theorem. We note the two properties needed from $h$ above are that $h$ be
decreasing and $\max_{a \in (0,1]} |ah'(a)| \leq 1$. The latter property in fact
follows from $h$ being of log-concave type and the normalization $|h'(1)|=1$. 

Motivated by the above, we consider the following simplification of
$\realdual_h(X)$, which we call the \emph{entropic dual}:  
\begin{equation}
\label{eq:entropic-dual}
\ent_h(X) \defeq \max_{\nu} \int_X \int_0^\infty h(\nu(B(x,\eps))) d\eps d\nu(x).
\end{equation}
This corresponds to the value $\nu$ using the nearly optimal guess for $\mu$
in~\eqref{eq:real-dual-value}, which is now readily computable. The following
direct corollary relates the value of the entropic dual to the actual dual.
\begin{corollary}
\label{cor:dual-to-ent}
\[
\realdual_h(X) \leq \ent_h(X) \leq 2\realdual_h(X) + \diam(X)/e.
\]
\end{corollary}

In terms of the additive error above, as mentioned at the beginning of the
section, $\diam(X)/4$ will be the trivial lower bound on $\realdual_h(X)=\primal_h(X)$.
Therefore, the right hand in \thmref{cor:dual-to-ent} is at most
$(2+4/e)\realdual_h(X)$ in the worst-case. In most interesting cases however, one
would expect $\primal_h(X)$ to be far from the trivial lower bound, in which case
one can think of the right hand side as $(2+o(1))\realdual_h(X)$. 

We are now ready to give the final simplified form of the dual whose value will
most directly relate to the value of packing trees: we define the \emph{simplified
dual} by
\begin{equation}
\label{eq:simplified-dual}
\dual_h(X) \defeq \max_{\nu} \min_{x \in X, \nu(x) > 0} \int_0^\infty
h(\nu(B(x,\eps)))d\eps. 
\end{equation}
Note that we restrict to the minimum of the points supported by $\nu$. These
will correspond to the potential leaf nodes in the packing tree. Furthermore,
the minimum in~\eqref{eq:simplified-dual} is in direct analogy to the minimum
cost of a path down a packing tree. 

Trivially, since we replaced the average by a minimum, we have that $\ent_h(X)
\geq \dual_h(X)$. We show that the reverse direction also holds up to additive
error. For a probability measure $\nu$ on $X$ and for any subset $S \subseteq
X$, satisfying $\nu(S) > 0$, define $\nu_{S}$ by
\begin{equation}
\label{eq:cond-measure}
\nu_S(A) \defeq \nu_{S}(A \cap S)/\nu(S), ~~~~ \forall A \subseteq X,
\end{equation}
i.e., $\nu_S$ is the conditional probability measure induced by $\nu$ on $S$.
The following lemma shows that one can easily convert a measure $\nu$ with large
$\ent_h$ value to one with large $\dual_h$ value via conditioning. 

\begin{lemma}
\label{lem:ent-to-simple}
For any probability measure $\nu$ on $X$, there exists
$S \subseteq \{x \in X: \nu(x) > 0\}$ such that
\[
    \int_X \int_0^\infty h(\nu(B(x,\eps)))d\eps \leq \min_{x \in S} \int_0^\infty h(\nu_S(B(x,\eps)))d\eps \;+\; \diam(X).   
\]
Furthermore, $S$ can be computed using at most $O(n^3)$ arithmetic
operations and evaluations of $h$.
\end{lemma}
The algorithm achieving the above is in fact very simple: we start with $S = \{x
\in X: \nu(x) > 0\}$, and iteratively kick out the element $x \in S$ with lowest
value as long as the above inequality is not met. 

Combining~\corref{cor:dual-to-ent} and~\lref{lem:ent-to-simple},  we obtain the following
relations between the dual program $\realdual_h(X)$ and the simplified dual
$\dual_h(X)$.

\begin{theorem} 
\label{thm:dual-to-simple}
\[
\realdual_h(X)-\diam(X) \leq \dual_h(X) \leq 2\realdual_h(X) + \diam(X)/e.
\]
Furthermore, given any probability measure $\nu$ on $X$, one can compute $S
\subseteq \{x \in X: \nu(x) > 0\}$ satisfying
\[
\int_X \int_0^\infty h(\nu(B(x,\eps)))d\eps d\nu(x) \leq
\min_{x \in S} \int_0^\infty h(\nu_S(B(x,\eps)))d\eps \;+\; \diam(X), 
\]
using at most $O(n^3)$ arithmetic operations and evaluations of $g$.
\end{theorem}

We are in fact not the first to examine the $\ent_h(X)$ and $\dual_h(X)$
programs. The analysis of solutions to $\ent_2(X)$ (i.e., $h=g$) in fact already
goes back all the way to Fernique~\cite{Fernique75}. Starting from a Gaussian
process $(Z_x)_{x \in X}$, Fernique examined the measure $\nu$ on $X$ satisfying
$\nu(u) = \Pr[Z_u = \max_{x \in X} Z_x]$, $\forall u \in X$, where we assume $X$
is finite and that the maximum is uniquely attained with probability $1$. For
this ``argmax'' measure $\nu$, it was shown that
\[
\E\left[\sup_{x \in X} Z_x \right] \asymp \int_X \int_0^\infty g(\nu(B(x,\eps)))d\eps d\nu(x),
\]  
where the inequalities $\lesssim$ and $\gtrsim$ where proven by
Fernique~\cite{Fernique75} and Talagrand~\cite{Talagrand87} respectively.  

The relationship between $\primal_h(X)$ and $\dual_h(X)$ was also already
studied by Naor and Mendel~\cite{MN11} as well as Bednorz~\cite{B12}. In
particular, for any continuous $h$ satisfying $\lim_{x \rightarrow 0^+} h(x) =
\infty$ (i.e., not necessarily of log-concave type),~\cite{MN11,B12} showed that
$\primal_h(X) \leq \dual_h(X)$. This was proved using Brouwer's fixed-point
theorem, which was used to find a measure $\mu$ where the quantities
$\int_0^\infty h(\mu(B(x,\eps)))d\eps$ are equal for all $x \in X$. Recalling that $\primal_h(X) = \realdual_h(X)$, this bound is stronger than $\realdual_h(X)-\diam(X) \leq \dual_h(X)$ in~\thmref{thm:dual-to-simple}. However, we do not require $\lim_{x \rightarrow 0^+} h(x) = \infty$ (e.g., $h(x) = 1-x$ is valid for us), which is crucially used to prove the existence of the above Brouwer measure.  Mendel
and Naor~\cite{MN11} further prove that $\dual_h(X) \lesssim \primal_h(X)$, for
any decreasing $h$ satisfying $h(x^2) \lesssim h(x)$. This is achieved by 
rounding any dual measure $\nu$ to what they call an \emph{ultrametric
skeleton}, which one can interpret as a very sophisticated analogue of a
packing tree which is agnostic to $h$. These skeletons are also used to derive
optimal bounds for the largest subset of a metric space embeddable into $\ell_2$
with small distortion (known as a non-linear Dvoretzky theorem). \cite{MN11}
asked whether one can improve their bound to $\dual_h(X) \leq (2+o(1))\primal_h(X)$,
where the factor of $2$ is tight up to $o(1)$ factors for an $n$-point
star-metric with $h=g$. Up to the additive constant (which is often $o(1)$
compared to $\primal_h(X)$) and the restriction that $h$ be of log-concave type,
\thmref{thm:dual-to-simple} resolves their question in the affirmative.

\subsubsection{Rounding Measures to Trees}

Towards proving our main theorem (\thmref{thm:constructive}), we show how to
round a measure $\rho$ to chaining and packing trees with values approximately
$\primal_h(\rho, X)$ and $\dual_h(\rho, X)$ respectively.  As remarked before,
the primal rounding strategy was already introduced by Talagrand \cite{T01}
himself. While the algorithm is simple, it is based on a not terribly intuitive
variant of ball partitioning, which is perhaps due to the structure of the
object he converts to (an admissible sequence). In the present work, we show
that Talagrand's basic greedy ball partitioning scheme with the functional
replaced by a measure, very transparently yields a construction of good chaining
trees. The greedy ball partitioning algorithm iteratively selects centers that
maximize the measure of balls of a smaller radius and removes a ball of a larger
radius centered at the previously chosen points. One does this until the removed
pieces form a partition and then proceeds recursively on those pieces. Here we
show that this can be implemented in near linear time in the input size which is
$O(n^2)$ for an $n$-point metric space.

\begin{theorem}\label{thm:lnet}
	There is a deterministic algorithm that given a probability measure $\rho$ on an $n$-point metric space $X$, finds a chaining tree $\CCT$ such that $\valsep(\CCT)  \lesssim  \primal_h(\rho, X)$  in $O(n^2\log n)$ time.  
\end{theorem}

On the dual side, as far as we are aware there were no rounding strategies to
compute packing trees starting from a measure $\rho$. The previous approaches for rounding \cite{Talagrand87, LT91, DLP12} were based on defining a functional that satisfies the previously mentioned ``super-chaining" inequality and used a greedy partitioning procedure based on the value of this functional. This analysis was rather delicate and somewhat mysterious. Moreover, the corresponding functionals in these instances were themselves solutions to optimization problem on the metric space, so implementing this strategy deterministically was rather slow. In fact, Ding, Lee and Peres \cite{DLP12} showed that using a carefully constructed dynamic program, one can implement the above strategy and compute a packing tree in polynomial time in the input size when the metric space is given as the input. Although, they don't specify a precise bound on the running time, one can directly infer a $O(n^4)$ deterministic running time for building a packing tree; with an additional observation, this can in fact be improved to a $O(n^3\log n)$ bound.%
\footnote{Their running time is $O(n^2)$ times the number of ``distance scales'' in the metric, meaning the number of dyadic intervals $[2^k, 2^{k+1})$ containing at least one distance $d(u,v)$ in the metric. 
    To bound the number of distance scales by $O(n \log n)$, 
    compute a minimum spanning tree $T$ in the complete graph describing the metric. 
    Consider any edge $e=\{v,w\}$ not in the tree, and let $m(v,w)$ denote an edge on the path between $v$ and $w$ in $T$ of maximum length. 
    Then $e$ has length at least the length of $m(v,w)$ since $T$ is an MST, but not more than $n$ times larger. That is, the distance scale of $e$ is in a range of size $O(\log n)$ of the distance scale of $m(v,w)$.
Now consider assigning each non-tree edge $\{v,w\}$ to $m(v,w)$; there are $O(\log n)$ distance scales assigned to each tree edge, so $O(n \log n)$ in total.}

In this work, we revisit the above approach and show that in fact, one can round a probability measure $\rho$ on the metric space to a packing tree with approximately the same value as $\dual_h(\rho, X)$. This has certain advantages over a rounding strategy using functionals as it can be implemented in near linear time in the input size. Moreover, this rounding algorithm is quite similar to the primal rounding algorithm and in our opinion clarifies why such a construction works. In particular, the basic strategy of choosing centers that maximize the measure of a smaller ball remains exactly the same --- one just selects smaller balls to add as children and recurses on them instead, followed by some post-processing.

\begin{theorem}\label{thm:sep}
	There is a deterministic algorithm that given a probability measure $\rho$ on an $n$-point metric space $X$, finds a $\frac1{10}$-packing tree $\CT$ such that $\dual_h(p, X) \lesssim \valsep(\CT)$ in $O(n^2\log n)$ time.  
\end{theorem}

\subsubsection{Proof of the Combinatorial Min-Max Theorem (\thmref{thm:constructive})}

Of course, without a way to compute measures which have almost optimal values for $\primal_h(\rho, X)$ and $\dual_h(\rho, X)$, the above rounding algorithms would not have been very useful. Fortunately, we can obtain almost optimal primal and dual measures $\mu$ and $\nu$ by solving the saddle point formulation \eqref{eq:saddle-point} of $\primal_h(X)$.  Plugging the measure $\mu$ in \thmref{thm:lnet} gives us a chaining tree $\CCT^*$ such that $\valsep(\CCT^*) \lesssim \primal(\mu, X) \asymp \primal_h(X)$.

On the dual side, the measure $\nu$ is not enough as it might not be a good solution to the approximate dual $\dual_h(X)$. However, using \thmref{thm:dual-to-simple}, one can find a set $S \subset X$ such that the probability measure $\nu_S$ obtained by conditioning $\nu$ on the set $S$ satisfies $\dual_h(\nu_S, X) \asymp \primal_h(X)$. Plugging the measure $\nu_S$ in \thmref{thm:sep}, gives us a packing tree $\CT^*$ satisfying $\dual_h(\nu_S, X) \lesssim \valsep(\CT^*)$. Combining the two yields that 
\begin{equation}\label{eq:gen-strong-duality}
 \valsep(\CCT^*) \lesssim \primal_h(X) \asymp \dual_h(\nu_S, X) \lesssim \valsep(\CT^*).
\end{equation}

Recall that the weak duality relation \eqref{eq:gen-weak-duality} between chaining and packing trees implies the reverse inequality for $\valsep(\CT) \lesssim \valsep(\CCT)$ for any chaining tree $\CCT$ and any packing tree $\CT$. Thus, \eqref{eq:gen-strong-duality} gives us the combinatorial min-max statement given by \thmref{thm:constructive}. Moreover, this can be made algorithmic by solving the saddle point formulation and using the algorithm to find the set $S$ given by \thmref{thm:dual-to-simple}. The details for solving the saddle point formulation are given in \appref{sec:alg} and the algorithm to find the set $S$ is presented in the proof of \thmref{thm:dual-to-simple}.

\subsection{Organization}

In \secref{sec:prelims}, we list some basic notation as well as the main
useful properties of chaining functionals of log-concave type. In
\secref{sec:convex}, we simplify the dual program $\realdual_h(X)$, proving
the main inequalities relating it to the entropic and simplified duals
$\ent_h(X)$ and $\dual_h(X)$ respectively. In \secref{sec:trees}, we give
near-linear time rounding algorithms from measures to chaining trees and packing
trees. In \secref{sec:gordon}, we give our deterministic construction of Johnson-Lindenstrauss
projections achieving Gordon's bound. In \appref{sec:alg}, we show how to use a
black-box solver to compute nearly-optimal primal and dual measures for the
saddle-point formulation of $\primal_h(X)$.

\section{Preliminaries}
\label{sec:prelims}

\paragraph{Notation.}

 Throughout this paper, $\log$ denotes the natural logarithm unless the base is explicitly mentioned. We use $[k]$ to denote the set $\{1,2,\dotsc, k\}$. For a vector $z \in \BR^n$, we will use $z_i$ or $z(i)$ interchangeably to denote the $i$-th coordinate of $z$. Given a probability measure $\mu$ over a set $X$, we use $\BE_{x \sim \mu}[f(x)]$ to denote the expectation of $f(x)$ where $x$ is sampled from $\mu$.
 
\subsection{Properties of Chaining Functionals} 
\label{sec:derivative}

Recall that we work with a chaining functional $h:(0,1] \rightarrow \R_+$
defined by $h(p) = F^{-1}(p)$, where $F(s) = \int_s^\infty f(x)dx$ and where $f$
is non-decreasing and continuous probability density on $\R_+$. We assume
throughout the rest of the paper that $h$ is of log-concave type, i.e., that $f$
is log-concave.   

The fundamental property of such functionals, that we make extensive use of, is
the following:  

 \begin{restatable}{proposition}{submultiplicative}
 	\label{prop:submultiplicative}
 	For a chaining functional $h$ of log-concave type, 
 	 $$h(ab)\leq h(a)+h(b) ~\text{ for }~ a, b\in (0,1].$$
 \end{restatable}

We note that the above property appears in~\cite{LT91} as the base assumption
for the chaining functionals they consider. The above shows that this condition
is very natural and applies to a wide variety of functionals.  

When working on a metric space $(X,d)$ with a chaining functional $h$, we
observe the following rescaling symmetry: for any constant $\beta > 0$, if we
replace $h$ by $\beta h$ (equivalently, the density function $f$ is replaced by
$z \to \beta f(\beta z)$) and $d(u,v)$ by $d(u,v)/\beta$ for all $u,v \in X$,
the values of the various quantities we consider remain unaffected.  In particular, if we
have an underlying process $(Z_x)_{x \in X}$ amenable to $X$ and $f$ as
in~\eqref{eq:tail-assumption}, the condition $\Pr[|Z_{x_1} - Z_{x_2}| \geq
d(x_1,x_2)s] \leq F(s)$ is identical for both scalings, and the values of the
various programs and trees we consider remain unchanged.

So from now on, we make the convenient choice that $f(0) = 1$.  This implies
that the (left) derivative $h'(1) = -1/f(0) = -1$. We maintain this normalization for the
remainder of the paper.

Given this normalization, the following useful bound is easy to show (we postpone the details to the appendix).

\begin{restatable}{proposition}{derivative}
\label{prop:derivative}
	For every $a \in (0,1]$, we have that
	 \[ -1 \le ah'(a) \le 0 \text{ and } h(a) \leq \log(1/a).\]
\end{restatable}

\section{Dual simplifications}
\label{sec:convex}

In this section, we provide the main arguments that allow for rounding a solution to $\realdual_h(X)$ to a solution to the simplified dual, via the entropic dual, as described in \secref{sec:simplifyingdual}.

\paragraph{Proof of \lref{lem:guess-mu}.}

Our goal is to prove that taking $\mu=\nu$ in \eqref{eq:real-dual-value} does not cause too much error. 
We will prove this on a ``per scale'' basis:

\begin{lemma}\label{lem:scale}
    For any probability measures $\mu$ and $\nu$ on $X$, and any $\eps >0$, we have 
    \[ \int_X h(\nu(B(x,2\eps)) d\nu(x) \le \int_X h(\mu(B(x,\eps))) d\nu(x) + \Ch/e.\]
\end{lemma}

\lref{lem:guess-mu} follows easily from this:
\begin{align*}
      \int_X \int_0^\infty h(\nu(B(x,\eps)))  d\eps d\nu(x) 
                &\le  \min_{\mu}\int_X \int_0^{\diam(X)} \Big( h(\mu(B(x,\eps/2))) + 1/e\Big)  d\eps d\nu(x)\\
                &\le  2\min_{\mu}\int_X \int_0^{\infty} h(\mu(B(x,\eps)))d\eps d\nu(x) + \diam(X)/e.
\end{align*}

\begin{proof}[Proof of \lref{lem:scale}]
    We first recast the statement in graph theoretic terms. Let us define a bipartite graph on the vertex set $X_1 \cup X_2$ where each $X_i$ for $i \in [2]$ is a copy of the index set $X$. We add an edge between two vertices $x_1 \in X_1, x_2 \in X_2$ if $d(x_1,x_2) \le \eps$ -- note that every vertex has an edge incident to it. Define the weight of a vertex $x \in X_i$ as $\mu(x)$ and the weight of a vertex $x \in X_2$ as $\nu(x)$. Let $E$ denote the set of edges in the graph, let $N(S)$ be the neighbors of a subset of the vertices $S \subset X_1 \cup X_2$ and let $N^2(S) = N(N(S))$. For brevity, we will write $N(x)$ instead of $N(\{x\})$ for singleton sets. We also use $\partial(x)$ to denote the set of edges incident on a vertex $x \in X_1 \cup X_2$. With this setup, the statement we want to prove is
    \begin{align}\label{eqn:theta}
    \ \sum_{x \in X_2} \nu(x) h(\nu(N^2(x))) \le  \sum_{x \in X_2} \nu(x) h(\mu(N(x)) + \Ch/e.
    \end{align}

To prove the above, we will use the following structural result which is a consequence of the theory of principal sequences of matroids~\cite{Fujishige2009}, applied to transversal matroids; 
we include a self-contained proof in \appref{sec:hall}.

\begin{restatable}{proposition}{matching}
  \label{clm:match}
  There exist sequences $\emptyset=S_0 \subset S_1 \subset \cdots \subset S_k = X_1$ and $0 < \beta_1 < \beta_2 < \cdots < \beta_k$, such that 
    \begin{enumerate}
        \item $\beta_i\mu(S_i \setminus S_{i-1}) = \nu( N(S_i) \setminus N(S_{i-1}))$ for all $i \in [k]$.
        \item For all $i \in [k]$ and $A \subseteq X_1 \setminus S_{i-1}$, we have that $\beta_i \mu(A)\le \nu(N(A) \setminus N(S_{i-1}))$.
    \end{enumerate}
\end{restatable}
\smallskip

The proposition above can be viewed as a kind of strengthening of Hall's theorem. For instance, if there is a fractional perfect matching between $\mu$ and $\nu$, i.e., a way of transporting mass distributed according to $\mu$ on $X_1$ along edges of the graph to yield precisely the distribution $\nu$ on $X_2$, then the claim will be satisfied with $k=1$, $S_1 = X_1$ and $\beta_1 = 1$, for in this case, $\mu(A) \leq \nu(N(A))$ for any $A \subseteq X_1$ (essentially the easy direction of Hall's theorem).  If there is no fractional perfect matching, Hall's theorem implies the existence of a set $S \subset X_1$ with $\nu(N(S)) < \mu(S)$. In the proposition, $S_1$ is the ``least matchable'' set: only a $\beta_1$ fraction of the mass in $S_1$ can be transported to $N(S_1)$. The full sequence is then obtained by removing $S_1$ and $N(S_1)$ and repeating on the remainder.

Taking the sequences guaranteed by the proposition, define $\tilde{\beta}_i \defeq \min\{\beta_i, 1\}$.

We split the left hand side of \eqref{eqn:theta} as follows: 
\begin{align*}
	\ \sum_{x \in X_2} \nu(x) h(\nu(N^2(x))) &= \sum_{i=1}^{k} \sum_{x \in N(S_i)\setminus N(S_{i-1})} \nu(x) h(\nu(N^2(x))).
\end{align*}
Note that for any $i \in [k]$, if $x \in X_2 \setminus N(S_{i-1})$, then $N(x) \subseteq X_1 \setminus S_{i-1}$, and hence, \pref{clm:match} implies that $\tilde{\beta_i} \mu(N(x)) \le \nu(N^2(x))$. 
Furthermore, using sub-multiplicativity of $h$ and \pref{prop:derivative}, we find that for any $i \in [k]$ and $x \in N(S_i) \setminus N(S_{i-1})$, 
\begin{align*}
 \ h(\nu(N^2(x))) = h\left(\mu(N(x)) \cdot \frac{\nu(N^2(x)}{\mu(N(x))}\right) &\le h(\mu(N(x))) + h\left(\min\left\{1, \frac{\nu(N^2(x))}{\mu(N(x))}\right\}\right) \\
\ &\hspace*{-3pt}\le h(\mu(N(x)) + h(\tilde{\beta}_i) \le h(\mu(N(x)) + \log\left(\frac1{\tilde{\beta}_i}\right).
 \end{align*}
 For the first inequality above, sub-multiplicativity is used when $\frac{\nu(N^2(x))}{\mu(N(x))}>1$; otherwise, the inequality follows because $h$ is decreasing and $h(1)=0$. 
 (The first inequality does still hold if $\mu(N(x)) = 0$, taking the minimum in the second term to have value 1 in this case, again because $h$ is decreasing.)

 The second inequality again uses that $h$ is decreasing, as well as that $\tilde{\beta}_i \le 1$ for every $i \in [k]$;
 and the final inequality uses \pref{prop:derivative}.

 Let $\ell$ be maximal such that $\tilde{\beta}_{\ell} < 1$; so $\beta_i = \tilde{\beta}_i$ for $i \leq \ell$.
Summing the last inequality over all $i \in [k]$ and $x \in N(S_i) \setminus N(S_{i-1})$, 
we obtain
\begin{align*}
    \  \sum_{x \in X_2} \nu(x) h(\nu(N^2(x))) &\le \sum_{x \in X_2} \nu(x) h(\mu(N(x))) +  \sum_{i=1}^{k} \nu(N(S_i) \setminus N(S_{i-1})) \cdot \log\left(\frac1{\tilde{\beta}_i}\right)\\
                                            \ &= \sum_{x \in X_2} \nu(x) h(\mu(N(x))) +  \sum_{i=1}^{\ell} \tilde{\beta}_i \mu(S_i \setminus S_{i-1}) \cdot \log\left(\frac1{\tilde{\beta}_i}\right)\\
                                          \ \ &\le \sum_{x \in X_2} \nu(x) h(\mu(N(x))) +  \left(\max_{\beta \in (0,1]} \beta  \log\left(\frac1{\beta}\right) \right) \sum_{i=1}^{\ell}  \mu(S_i \setminus S_{i-1}),
\end{align*}
where the second line follows from \pref{clm:match}. 
From this \eqref{eqn:theta} readily follows as $\sum_{i=1}^{\ell} \mu(S_i \setminus S_{i-1}) \le \sum_{i=1}^{k} \mu(S_i \setminus S_{i-1}) = \mu(X_1)=1$ and $\max_{\beta \in (0,1]} \beta  \log\left(\frac1{\beta}\right)  = 1/e$.

\end{proof}

\paragraph{Proof of \lref{lem:ent-to-simple}.}
For convenience, define 
\[ \HH(\mu, t) \defeq \int_0^\infty h(\mu(B(t,\eps)))d\eps \qquad \text{ and } \qquad \HH(\mu, \nu) \defeq \int_X \HH(\mu, t) d\nu(t). \]
Start by setting $S = \{ x \in X : \nu(x) > 0 \}$.
Consider the following greedy algorithm:

\begin{quote}
As long as $\HH(\nu_S, \nu_S) > \min_{x \in S} \HH(\nu_S, x) + \diam(X)$, choose
$s \in S$ so that $\HH(\nu_S, s)$ is minimized, and remove $s$ from $S$.
\end{quote}

Note that $\HH(\nu_S, \nu_S) \le \min_{x \in S} \HH(\nu_S, x) + \diam(X)$ when this terminates, since it is vacuous for $S=\emptyset$.

We will now show that $\HH(\nu_S, \nu_S)$ can only increase during the progression of the algorithm.
This suffices to prove the lemma, since then upon termination
\[ \HH(\nu, \nu) \leq \HH(\nu_S, \nu_S) \le \min_{x \in S} \HH(\nu_S, x) + \diam(X). \]

So, consider a moment in the algorithm where $s \in S$ is about to be removed from $S$, yielding $S' \defeq S \setminus \{s\}$. From our choice of $s$, 
\begin{equation}\label{eq:sprop}
    \HH(\nu_S, \nu_S) > \HH(\nu_S, s) + \diam(X).
\end{equation}
Let $\alpha = 1/(1-\nu_S(s))$; so $\nu_{S'}(t) = \alpha \nu_S(t)$ for $t \neq s$, and $\nu_{S'}(s) = 0$.
Thus 

\begin{claim}\label{clm:entinc1}
    For every $t \in X$,
    \begin{equation*}
        \HH(\nu_{S'}, t) \ge \HH(\nu_{S}, t) - \diam(X) \cdot (\alpha - 1).
    \end{equation*}
\end{claim}
\begin{proof}
    Recall that $h$ is convex and satisfies $-1 \le ah'(a) \leq 0$ for any $a \in (0,1]$ by \propref{prop:derivative}.
    Thus for any $z \in (0,1]$,
    \begin{equation}\label{eq:hbeh}
        h(z) \ge h(z/\alpha) + z(1 - 1/\alpha)h'(z/\alpha) = h(z/\alpha) -  (\alpha-1)\bigl|(z/\alpha) h'(z/\alpha)\bigr| \ge  h(z/\alpha) -  (\alpha-1).
    \end{equation}

    Now notice that for any $T \subseteq S$, we have that $\nu_S(T) \ge \nu_{S'}(T)/\alpha$. 
    Therefore, since $h$ is decreasing, 
    \begin{align*}
        \HH(\nu_S, t) &= \int_0^{\diam(X)} h(\nu_S(B(t,\eps)))d\eps \\
                      &\le  \int_0^{\diam(X)} h(\nu_{S'}(B(t,\eps)) / \alpha))d\eps\\ 
                    &\leq \int_0^{\diam(X)} h(\nu_{S'}(B(t,\eps)))d\eps + \diam(X)\cdot (\alpha - 1) ~=~ H(\nu_{S'} ,t) + \diam(X) \cdot (\alpha-1), 
    \end{align*}
    where the last inequality follows from \eqref{eq:hbeh}.
\end{proof}

We can now compare $\HH(\nu_{S'}, \nu_{S'})$ with $\HH(\nu_S, \nu_S)$: 
    \begin{align*}
        \HH(\nu_{S'}, \nu_{S'}) &\geq \HH(\nu_S, \nu_{S'}) - \diam(X)\cdot (\alpha-1) \qquad \qquad \qquad \qquad \qquad \qquad \qquad \:\: \text{(by \clmref{clm:entinc1})}\\
                                             &= \tfrac{1}{1-\nu_S(s)}\left(\HH(\nu_S, \nu_S) - \nu_S(s)\HH(\nu_S, s)\right) - \tfrac{\nu_S(s)}{1-\nu_S(s)}\cdot \diam(X)\\
                                             &\geq \tfrac{1}{1-\nu_S(s)}\bigl(\HH(\nu_S, \nu_S) - \nu_S(s)(\HH(\nu_S, \nu_S) - \diam(X))\bigr) - \tfrac{\nu_S(s)}{1-\nu_S(s)}\cdot \diam(X)  \:\: \text{(by \eqref{eq:sprop})}\\
                                             &= \HH(\nu_S, \nu_S),
    \end{align*}
which finishes the proof of \lref{lem:ent-to-simple}.

    \paragraph{Runtime Analysis.}  It is easily seen that the algorithm can be implemented in $O(n^3)$ time. First, by pre-processing the input, we may assume that the pairs of points are sorted by their pairwise distances --- this only adds an additional overhead of $O(n^2 \log n)$. Then,  as in each iteration we remove one element, there are $O(n)$ iterations to compute the final set $S'$. Furthermore, in each of these iterations, we are required to compute a minimizer $s$ of $\HH(\nu_S, x)$. 
This takes $O(n^2)$ time since for each $x \in S$, one can compute the measure of all possible balls $B(x,\eps)$ in $O(n)$ time using a straightforward dynamic program.

\section{Rounding Measures to Trees}\label{sec:trees}

In this section we give algorithms to round measures  to chaining and packing trees. To introduce the rounding algorithm for chaining trees, we first introduce \emph{labelled nets}. These are another kind of trees on the primal side that will be much easier to construct and these can easily be converted to chaining trees.

\subsection{Labelled Nets}\label{sec:lab-nets}

\begin{definition}[Labelled Nets] 
\label{def:lab-nets} Let $\alpha \in (0, \frac{1}{10}]$. An $\alpha$-labelled net $\CL$ on a finite metric space
$(X,d)$ is a rooted tree on subsets of $X$, with root node $R = X$,
together with labelings $m: \CL \rightarrow \Z_{\ge 0}$ and $\sigma: \CL \to \N$. Every non-leaf node $V \in \CL$ has a set of children $C_1,\ldots, C_k \subsetneq V$
which form a partition of $V$. We also require that every leaf node $V \in \CL$ is a singleton, i.e., $V=\{x\}$ for some $x \in X$. \\

We enforce the follow properties on $\CL$: 
\begin{enumerate}
\item \label{eq:diam} For $V \in \CL$,  the diameter $\diam(V) \leq \alpha^{m(V)} \cdot \diam(X)$. 
\item For any child $C$ of $V \in \CT$, the label $m(C) \ge m(V) + 1$. 
\item The children of a vertex are ordered and $\sigma$ gives the ordering of a vertex with respect to its parent. In particular, if the children of a vertex $V$ are $C_1, \ldots, C_k$, then $\{\sigma(C_i) \mid  i \in [k]\} = [k]$. 
\end{enumerate}

The value of an $\alpha$-labelled net $\CL$ with respect to a chaining functional
$h$ is defined as 
\begin{equation}
\label{eq:lab-net-value}
\valsep(\CL) ~\defeq ~ \sup_{x \in X} \sum_{(V,W) \in \CP_x}
\alpha^{m(V)} \cdot \diam(X) \cdot h(1/\sigma(W)),
\end{equation}
where $\CP_x$ is the unique path from the root $R$ to the leaf $\{x\}$, and the sum is over all the edges on the path where we write $(V,W)$ to denote the edge from $V$ to its child $W$. We use
the shorthand $\val_2(\CL)$ to denote the value with respect to the Gaussian
functional $g$. 
\end{definition}

Observe from the above definition that the diameters decrease by at least a factor of $\alpha$ as one goes down the tree. Furthermore, it is also easily seen that $\valsep(\CL) \ge  \diam(X) \cdot h(1/2)$ for any labelled net $\CL$. 

\paragraph{Labelled Nets to Chaining Trees.} 

Given an $\alpha$-labelled net $\CL$ for a metric space $(X,d)$, in $O(n^2)$ time where $n=|X|$ one can compute a chaining tree $\CCT$ that has the same value up to constant factors. For this consider the following algorithm. Below for notation, we write $\CP_V$ to denote the unique path from the root to the vertex $V$ in the labelled net $\CL$.

Let $R$ be the root of the labelled net $\CL$ and initialize the root $w$ of the chaining tree $\CCT$ to be any arbitrary point in $R$. Then, we do the following recursively:
for every child $C$ of $R$ in the labeled net, pick an arbitrary point $t \in C$ that has not been picked before during the previous steps and add $t$ as a child of $w$. For the edge $e=(w,t)$, we define $p_e = \frac32\pi^{-2} \cdot 2^{-m(C)} \cdot \prod_{V \in \CP_C \setminus \{R\}} 1/(2\sigma(V))^2$ and $l_e = d(w,t) \cdot h(1/p_e)$. Then, we recurse on each child $C$ until all the nodes of $\CL$ or all the elements of $X$ are exhausted. 

Note that as $\sum_{i=1}^\infty (1/{i^2}) = \pi^2/6$, it follows that the sum of probability labels of all the edges is at most $1/2$, so the above algorithm produces a valid chaining tree. Furthermore, as a labelled net gives a laminar family of sets and hence is always of size $O(n)$, it is easily seen that the above can be implemented in $O(n^2)$ time. 

Next we argue that the value of the chaining tree is at most a constant factor larger than the value of the labelled net.

\begin{lemma}\label{lem:chainingtree} The value of the chaining tree $\CCT$ constructed by the above algorithm on the labelled net $\CL$ satisfies $\valsep(\CCT) \lesssim \valsep(\CL)$. 
\end{lemma}
\begin{proof}
Consider an arbitrary $x \in X$ and let $\CP'_x = (t_0, t_1, \cdots, t_k = x)$ be the path to the leaf $x$ in the chaining tree $\CCT$. By construction, there is a path $(V_0:=X, \cdots, V_k)$ in the labelled net $\CL$ such that $t_i \in V_i$ for $i \in \{0\} \cup [k]$.  It follows that $d(t_i, t_{i+1}) \le \alpha^{m(V_i)} \cdot \diam(X)$ by the definition of the labelled net.  Therefore, using that $m(V_i) \ge i$ and the fact that $h$ is decreasing and sub-multiplicative, we get that 
\begin{align*}
	\ \sum_{e \in \CP'_x} l_e &\le \sum_{i=1}^k \alpha^{m(V_{i-1})} \cdot \diam(X) \cdot  h\left(\frac1{8\cdot 2^{m(V_i)}} \cdot \prod_{j=1}^i\frac1{\sigma(V_j)^2}\right)  \\
	\ & \le \sum_{i=1}^k \alpha^{m(V_{i-1})} \cdot \diam(X) \cdot  h\left(\prod_{j=1}^i \frac1{\sigma(V_j)^2}\right) +\sum_{i=0}^{k-1} \alpha^{m(V_i)} \cdot \diam(X)  \cdot m(V_i) \cdot h\left(\frac1{2}\right) + h\left(\frac{1}{8}\right)\diam(X) \\
	\ & \le \sum_{i=1}^k 2\left( \sum_{j=i}^k \alpha^{m(V_{j-1})} \right) \cdot \diam(X) \cdot  h\left(\frac1{\sigma(V_i)}\right) + 4\sum_{i=0}^{k-1} \alpha^{m(V_i)} \cdot \diam(X)  \cdot m(V_i) \cdot h\left(\frac1{2}\right) ,
\end{align*}

Since $m(V_i)$ is increasing as one goes down the path and $\sum_{i\ge 0}i\cdot  \alpha^i  \le  \frac13$ for $\alpha \le \frac1{10}$, we have
\begin{align*}
	\ \sum_{e \in \CP'_x} l_e & \le \sum_{i=1}^k 4\alpha^{m(V_{i-1})} \cdot \diam(X) \cdot  h\left(\frac1{\sigma(V_i)}\right) + 4\diam(X)   \cdot h\left(\frac1{2}\right),
\end{align*}

The second term above is at most $4\valsep(\CL)$, and as $\{x\}$ was an arbitrary point in $X$, the above implies that $\valsep(\CCT) \le 8 \valsep(\CL)$. 
\end{proof}

\subsection{Rounding Measures to Labelled Nets}

To describe the greedy partitioning algorithm, we first preprocess the input metric space as follows.

\paragraph{Preprocessing.} We assume that for every $x \in X$, we have a list of all its neighbors sorted according to the distance from $x$. This can be done in $O(n^2 \log n)$ time. Sorting this list also allows to compute for every $x \in X$, the distinct breakpoints at which the balls $B(x, \eps)$ change -- the total time to compute all the breakpoints for every $x \in X$ is  $O(n^2)$. With this information at hand, using binary search and a straightforward dynamic program, one can pre-compute the measure of all the balls $\rho(B(x,\eps))$ for the different breakpoints $\eps$ for every $x \in X$. This takes $O(n)$ time for each $x$, so $O(n^2)$ time overall. The total time for preprocessing is $O(n^2 \log n)$.\\

We can now describe the greedy partitioning procedure to create a labelled net from a given probability measure $\rho$. To simplify the presentation, we remove the requirement that the child of a node be a proper subset of the node --- in particular, we allow a node to have a single child.
Such nodes can always be removed by simple post-processing.  

\paragraph{Greedy Ball Partitioning.}
We initialize $R = X$ and $m=m(R)=0$. We do the following:
\begin{enumerate}[label=(\alph*)]
	\item Initialize $i=1$ and let 
	$$t_{i} = \argmax_{x \in R \setminus S_{i-1}} \rho(B(x, \frac12\alpha^{m+2} \cdot \diam(X))) ~~\text{ and }~~ A_i = B(t_i, \frac12\alpha^{m+1} \diam(X)) \setminus S_{i-1},$$
	where $S_{i-1} = \cup_{j=1}^{i-1} A_j$. Increment $i$ until $A_i$'s form a partition of $R$. We refer to $t_i$ as the center of the node $A_i$.
	\item Add $A_i$'s as the children of $R$ with the ordering $\sigma(A_i)=i$ and $m(A_i)=m+1$. Then, for each $A_i$ that is not a singleton set, recurse by executing steps (a) and (b) with $R=A_i$ and $m=m(A_i)$.
\end{enumerate}
	
From the description of the algorithm, it is obvious that the requirements on the diameter of each node is satisfied, so the above algorithm constructs a valid labelled net.	
	
We stress that the sets $A_i$ correspond to removing a ball of radius $\frac12 \alpha^{m+1}  \cdot \diam(X)$ from the left over points in the current set, while the maximizers $t_i$'s are chosen by optimizing over a much smaller ball of radius $\frac12\alpha^{m+2} \cdot \diam(X)$. From this it also follows that the balls $B(t_i,\frac12\alpha^{m+2})$ are disjoint for different values of $i$ since the pairwise distance between $t_i$ and $t_j$ is at least $\frac12 \alpha^{m+1} \cdot \diam(X)$. Another very useful fact to observe is that the centers $t_1, \cdots, t_k$ chosen at a particular step are arranged in decreasing order according to $\sigma$, in particular. 
\[  \rho(B(t_{\sigma(1)}, r))  \ge \rho(B(t_{\sigma(2)}, r)) \ge \cdots \ge \rho(B(t_{\sigma(k)}, r)), \]
 where $r$ is the radius of the balls we optimized over while choosing $t_i$'s.

\paragraph{Running Time.} 

We store every node of the tree as a list of elements. To avoid duplication of the same sets, which may happen in the description of the above algorithm, we first describe a small modification to the above algorithm. The same subset $S \subseteq X$ appears many times only when $V_1, \ldots, V_{k}$ are nodes of the tree corresponding to the same subset $S$ and each $V_{i+1}$ is the only child of $V_i$ for $i \in [k-1]$. In this case the bound on the diameter of $S$ decreases by a factor of $\alpha$ in each iteration until we are within a factor of $\alpha$ of $\diam(S)$. In this case, assuming that the sequence terminates with $V_k$, where $V_k$ has more than one child, we can just add $V_k$ as the only child of $V_1$ and record the labels $m(V_1)$ and $m(V_k)$. By adding a simple check in step (b) and computing the diameter of $V_1$ while processing its parent $U$, this can be done in $O(|U|)$ time for each such set $V_1$ using the pre-computed data.

Next we show by induction on $|R|$ that the running time of a call to Greedy Ball Partitioning on $R$ is bounded by $c|R|^2 \log n$, for a constant $c>0$. For $|R|=1$ the algorithm takes constant time, proving the base case. Assume that the claim holds for $|R|<k$, consider a recursive call with $|R|=k$. Now a partition $\{A_1,\ldots, A_q\}$ is generated greedily. Every choice of $t_i$ and the construction of the corresponding $A_i$ takes at most $O(|R| \log n)$ steps since the measure of the balls is pre-computed (we assume it is stored in a sorted order). As described above we can assume that the partition consists of at least two elements, i.e. $q\geq 2$. A recursive subcall is executed on $A_i$ for every $i$ that by our induction hypothesis takes $c|A_i|^2\log n$ steps. So the total running time of the execution is 
$$c\log n \cdot (q|R| + \sum_{i=1}^q|C_i|^2)=c\log n \cdot \sum_{i=1}^q(q|C_i|+|C_i|^2)\leq c\log n \cdot(\sum_{i=1}^q|C_i|)^2=c|R|^2\log n,$$
where the inequality follows because $q\geq 2$ and $|C_i | \geq 1$ for all $i$. So each call to the algorithm takes at most $c|R|^2 \log n$ time.

A call with $R:=X$ and $|X|=n$ takes $O(n^2\log n)$ time. The preprocessing steps also need to be executed once and this takes $O(n^2\log n)$ time, so the total running time is $O(n^2\log n)$.

\paragraph{Value of the Labelled Net.} To complete the proof of \thmref{thm:lnet}, we show that the value of the labelled net is at most $O(\primal_h(\rho, X))$.

\begin{lemma}
	The value of the labelled net $\CL$ constructed by the greedy ball partitioning algorithm on the measure $\rho$ satisfies $\valch(\CL) \lesssim \primal_h(\rho, X)$.
\end{lemma}
\begin{proof}	
   Consider the path from the root to an arbitrary leaf $\{x\}$.  Since $h(1)=0$, any node that is the only child of its parent does not contribute to $\valch(\CL)$ and so let $(V_{0} =X, V_{1}, \cdots, V_{k} = \{x\})$ be all the nodes of this path (except the root and leaf) which have at least one sibling. Moreover, for for $i \in \{0\} \cup [k]$, let $m_i = \depth(V_i)$ and $r_i = \frac12 \alpha^{m_i} \cdot \diam(X)$ denote the bound on the radius  (one half of the diameter) of $V_{i}$. Then, we shall prove that for every $i \in \{0\} \cup [k]$ and for $\eps \in (\alpha^3  r_{i-1}, \alpha^2  r_{i-1})$, the following holds
\begin{align}\label{eqn:2}
\ \rho(B(x,\eps))  \le \rho(B(t_i, \alpha^2 r_{i-1})) \le 1/\sigma(V_{i}),
\end{align}
where $t_i$ is the center of $V_{i}$.  As $h$ is decreasing, using the above gives us
 \begin{align*}
 	\ \alpha^2(1-\alpha) \cdot r_{i-1} \cdot h\left(\frac{1}{\sigma(V_{i})}\right) \le \int\limits_{\alpha^3  r_{i-1}} ^{\alpha^2 r_{i-1}} h(\rho(B(x,\eps)))d\eps.
 \end{align*}
Summing the above inequality over all $i \in \{0\} \cup [k]$ gives us that 
 $$ \frac{1}{2} \cdot \alpha^2(1-\alpha) \cdot \sum_{i=1}^{k} \alpha^{m_{i-1}} \cdot \diam(X) \cdot  h\left(\frac{1}{\sigma(V_{i})} \right) \le \int_0^{\diam(X)}h(\rho(B(x,\eps)))d\eps,$$
since we are only integrating over a larger interval. This gives the statement of the lemma as $x$ is arbitrary. 
 
To see \eqref{eqn:2}, let $U$ be the parent node of $V_{i}$ and let $s_j$'s be the centers chosen by the algorithm while processing $U$. Note that since the depth of $U$ is $m_{i-1}$, the center $s_j$'s are chosen to maximize the measure of the balls of radius $\frac12 \alpha^2 \cdot \alpha^{m_{i-1}} \cdot \diam(X) = \alpha^2 r_{i-1}$ and by construction they satisfy $\rho(B(s_j, \alpha^2 r_{i-1})) \ge \rho(B(s_{j+1}, \alpha^2 r_{i-1}))$ for every $j$. 

Given the above, the first inequality in \eqref{eqn:2} follows as $t_i$ is one of the centers and $x$ is a candidate maximizer while choosing any of the centers. Finally, for the second inequality in \eqref{eqn:2},  let $t_i = s_{j_*}$ for some index $j_*$. As the balls of radius $\alpha r_i$ are disjoint for distinct centers $s_j$'s and their measures are decreasing, by averaging, we must have $\rho(B(t_i, \alpha^2 r_{i-1})) = \rho(B(s_{j_*}, \alpha^2 r_{i-1})) \leq 1/j_*=1/\sigma(V_{i})$ by the definition of the ordering $\sigma$. This gives us \eqref{eqn:2} and completes the proof of the lemma.

\end{proof}

\subsection{Rounding Measures to Packing Trees}

In this section, we show how to round a measure $\rho$ to an $\alpha$-packing tree where $\alpha = 1/10$. To describe the algorithm, we assume that we have preprocessed the input as in the previous section. Furthermore, it will be convenient to extend $h$ to the interval $(0,\infty)$ by defining $h(a) = 0$ for $a>1$. It is easily seen that $h(ab) \le h(a) + h(b)$ holds for all $a,b > 0$ with this definition.

\newcommand{\CC}{\mathcal{C}}

\paragraph{Greedy Separated Ball Partitioning.} 
If $\dual(\rho,X) \ge 60\alpha^{-2} \cdot \diam(X) \cdot h(1/2)$, then we can output a trivial packing tree $\CT$ by adding two maximally separated elements as leaves of the root node $X$. Such a tree has value $\diam(X) \cdot h(1/2)$ so it satisfies $\dual_h(X) \lesssim \valsep(\CT)$. Thus, in the sequel we assume that $\dual(\rho,X)$ is larger than the above bound.

We first construct an auxiliary tree $\CT'$ where the children of each node are separated but the tree does not correspond to a laminar family of sets. Initialize $R=X$ to be the root of the tree $\CT'$ and $m=m(R)=0$. We do the following:
\begin{enumerate}[label=(\alph*)]
 \item Initialize $i=1$ and let 
	$$t_{i} = \argmax_{x \in R \setminus S_{i-1}} \rho(B(x, \frac14\alpha^{m+2} \cdot \diam(X)))  ~~\text{ and }~~ A_i = B(t_i, \frac14\alpha^{m+1} \cdot \diam(X)) \setminus S_{i-1},$$
where $S_{i-1} = \cup_{j=1}^{i-1} A_j$. Increment $i$ until $A_i$'s form a partition of $R$ and let $A_1, \cdots, A_{i_*}$ be the partition. We refer to $t_i$ as the center of the node $A_i$.
	
	\item By relabeling the indices, order the $A_i$'s such that $\rho(A_1)\geq \rho(A_2) \geq \cdots \geq \rho(A_{i_*})$ and for $i \in [i_*]$, let 
	$$B_i = B(t_i, \frac14 \alpha^{m+2}\cdot \diam(X)) ~~\text{ and }~~ L = \min_{i \in [i_*]} \left\{ i ~\bigg|~ \frac{\rho(A_i)}{\rho(R)} \geq \frac{6}{\pi^2}\cdot \frac1{i^2}\right\}.$$
	Note that $\sum_{i=1}^\infty 1/{i^2} = \pi^2/6$, so there always exists such an $L$. 
		\item \begin{enumerate}[label=\roman*.]
		
		\item \begin{sloppypar} If there is an $l \in [L]$ with $h(\rho(B_l)/\rho(A_l)) \geq 4\alpha^{-2} h (\rho(A_l)/\rho(R))$, then add $A_l$ as the sole child of $R$ in the separated tree with $m(A_l)=m+1$ and recurse by executing steps (a) and (b) with $R=A_l$ and $m = m(A_l)$.  
		\end{sloppypar}
		
		\item Otherwise, for each $l \in [L]$, add $B_l$  as a child of $R$ in the separated tree with $m(B_l)=m+2$. Then, for each $B_l$ that is not a singleton set, recurse by executing steps (a) and (b) with $R=  B_l$ and $m = m(B_l)$. 
		\end{enumerate}
\end{enumerate}

The above does not necessarily produce a laminar family of sets. So, for every node $V$ in the above tree $\CT'$, add to it all the elements that are present in the union of its children but are not in $V$. Furthermore, if a node $V$ only has a single child $C$, we replace $V$ with $C$ and keep doing this until every node in the tree has more than one children. This will be our final packing tree $\CT$.  The labeling function $\chi$ is defined in terms of $m$ as $\chi(R) = m(R) = 0$ for the root $R$ and $\chi(V) = m(V)+1$ for every other node $V$ that is not the root of $\CT$. Note that if we replace $V$ with $C$ the $\chi$ label we keep is the label for $C$.

\paragraph{Valid Packing Tree.} We now argue that the above algorithm produces a valid packing tree $\CT$ --- in particular, the nodes of the tree give a laminar family of sets, with appropriate diameter bounds and the children of each node are appropriately separated. 

First, we claim that the last step where we convert $\CT'$ to the tree $\CT$ by adding missing elements does not increases the diameter of each node by too much. In particular, let $V'$ be a node in $\CT'$ with center $t$ and radius $r = \frac14\alpha^{m(V')} \cdot \diam(X)$. Let $V$ be the corresponding node in the tree $\CT$.  Then, we claim by induction that $V \subseteq B(t, r/(1-\alpha))$. The base case at the bottom most level is easily seen. For the inductive step, there are two cases, in case $V'$ only has a single child, it must be that $V \subseteq B(t, r + \alpha r/(1-\alpha)) = B(t, r/(1-\alpha))$. Otherwise, $V'$ has multiple children and then $V \subseteq B(t, r + 2\alpha^2 r/(1-\alpha)) \subseteq B(t, r/(1-\alpha))$ since $\alpha = 1/10$. It thus follows that the diameter of any node $V$ in $\CT$ is at most $2r/(1-\alpha) \le \alpha^{m(V)} \cdot \diam(X)$. Since in the tree $\CT$ we remove all nodes with a single children, for each node $C$ with parent node $V$, we have that $m(C) \ge m(V) + 2$ and hence, by definition of $\chi$, we have that $\diam(C) \le \alpha^{\chi(V)+1} \cdot \diam(X)$. 

Next, we argue the separation. Suppose $V'$ has multiple children, let $U'_1$ and $U'_2$ be any two distinct children with centers $t_1$ and $t_2$. Then, by construction, we have that $U'_i$ is contained in a ball of radius $\alpha^2 r$ around $t_i$ and $d(t_1, t_2) \ge \alpha r$. It follows that the corresponding nodes $U_1$ and $U_2$ in the final separated tree $\CT$ satisfy $U_i \subseteq B(t_i, \alpha^2r/(1-\alpha))$ and since $\alpha \le \frac1{10}$, it follows that
$$d(U_1, U_2) \ge \alpha r - \frac{2\alpha^2}{1-\alpha} \cdot r = \frac{\alpha(1-3\alpha)}{1-\alpha}\cdot r \ge \frac{1}{10} \cdot \alpha^{m(V)+1}\cdot \diam(X) = \frac1{10} \alpha^{\chi(V)} \cdot \diam(X).$$ 

From the above, it also follows that no two children of $V$ can intersect in $\CT$, so the tree indeed gives a laminar family of sets with appropriate separation and diameter bounds. Thus, it is a valid packing tree.

\paragraph{Runtime Analysis.} 

We store every node of the tree as a list of elements and assume access to the same pre-processed data as in the case of labelled nets. Also, as in the case of labelled nets, we also assume that the same set is never duplicated, which can be ensured in $O(|U|)$ time while processing a node $U$ by computing the diameter. In addition, the post-processing of $\CT'$ to obtain $\CT$ can also be done in $O(n^2)$ time, so in the sequel, we analyze the running time to build $\CT'$.

For the analysis, we define $P(A):=A$ if $|A|=1$ and $P(A):=\{x\in X:d(x,A) \leq \frac14\alpha^{\kappa(A) + 2} \cdot \diam(X) \}$, where $\kappa(A):=\min\{k\in \mathbb{N} \mid \diam(A) \leq \alpha^k \cdot \diam(X) \}$.
We will now show with induction on the height of the output tree that the running time of a call to Greedy Ball Partitioning on $R$ is bounded by $c\log n\cdot (2|R|^2 + |P(R)|^2)$ for a constant $c>0$. For $|R|=1$ the algorithm takes constant time, which proves the base case.

Assume that the claim holds for all calls that output a tree of height smaller than $k$, and consider a call that returns a tree of height $k$.
A partition $\{A_1,\ldots, A_{i_*}\}$ is generated greedily. Every choice of $t_i$ and the construction of the corresponding $A_i$ takes at most $O(|R|\log n)$ steps  since the measure of the balls is pre-computed (we assume it is stored in a sorted order). Moreover, as described above we can assume that $i_* \ge 2$.

Now there are two cases: If (c-\textrm{i}.) is executed then only one recursive call is done with $R:=A_l$ for some $l \in [L]$. The time for processing $A_l$ is at most $c\log n\cdot  i_*|R| \le c\log n|R||R\setminus A_l|$ and by the induction hypothesis the running time of the recursive call is bounded by $c\log n (2|A_l|^2+|P(A_l)|^2)$, so ignoring the $c\log n$ multiplicative factor, the overall running time is at most
$$2|R||R\setminus A_l|+ 2|A_l|^2 +|P(A_l)|^2 \leq 2(|A_l|+|R\setminus A_l|)^2+|P(A_l)|^2 = 2|R|^2+|P(A_l)|^2.$$
Because $P(A_l)\subseteq P(R)$, the running time is bounded by $c\log n\cdot (2|R|^2+|P(R)|^2)$.

Otherwise, a recursive subcall is executed on $B_i$ for every $i \in [L]$ that by our induction hypothesis takes $c\log n\cdot (2|B_i|^2+|P(B_i)|^2)$ steps. Note that the $B_i$'s as well as $P(B_i)$'s are pairwise disjoint sets as they are sufficiently separated as we argued while showing that we obtain a valid packing tree. Let $K:=|R \setminus \cup_{i=1}^{i_*} B_i|$. The total running time of the execution is at most $c\log n \cdot (i_*|R| + \sum_{i\in L}2|B_i|^2 + \sum_{i\in L}|P(B_i)|^2)$. It is easily seen that $\sum_{i\in L} P(B_i)^2 \le P(R)^2$ and as for the other terms, we have that
\begin{align*}
 \ i_*|R| + 2\sum_{i \in L} |B_i|^2 &\le i_*K + i_*\sum_{i\in [{i_*}]} |B_i| +  2\sum_{i \in [{i_*}]} |B_i|^2 \\
 \ & \le 4K\sum_{i\in {[i_*]}} |B_i|  + 2 \sum_{\substack{i,j\in {[i_*]} \\i \neq j}} |B_i||B_j| +   2\sum_{i \in {[i_*]}} |B_i|^2 \le 2\left(\sum_{i \in {[i_*]}} |B_i|+K\right)^2= 2|R|^2,
\end{align*}
where we used that $i_*\geq 2$ and $|B_i|\geq 1$ for all $i \in [i_*]$. So, a call to the algorithm takes at most $c\log n \cdot(2|R|^2+|P(R)|^2)$ time.

A call with $R:=X$ and $|X|=n$ takes $O(n^2\log n)$ time. The preprocessing steps also need to be executed once and this takes $O(n^2\log n)$ time, so the total running time is $O(n^2\log n)$.

 \paragraph{Value of Packing Tree.}  To complete the proof of \thmref{thm:sep}, we show that the value of the packing tree is at least $\Omega(\dual_h(\rho, X))$.

\newcommand{\Xin}{X_{\mathsf{in}}}
\newcommand{\rt}{\mathsf{root}}

\begin{lemma} \label{lem:sep} The value of the  packing tree $\CT$ constructed by the greedy separated ball partitioning algorithm on the measure $\rho$ satisfies $\valsep(\CT) \gtrsim \dual_h(\rho, X)$. 
\end{lemma}
\begin{proof} 
First, we note that it suffices to show that $\valsep(\CT') \gtrsim \dual_h(\rho, X)$, where $\CT'$ was the tree before adding all the missing elements to obtain $\CT$ and 
$$\valsep(\CT') \defeq ~\inf_{x \in {\rm leaf}(\CT)} \sum_{V
\in \CP_x \setminus \{x\}}
\alpha^{m(V)} \cdot \diam(X) \cdot h(1/\deg(V)).$$

To see this first notice that $\valsep(\CT)$ depends only on the number of children and $\chi$ values. Furthermore, nodes with a single child do not contribute to $\valsep(\CT')$ as $h(1)=0$ and even though $\CT'$ does not give a laminar family of sets, the transformation from $\CT'$ to $\CT$ does not change the set of leaves, only increases the diameter by a constant factor and preserves the degree for each node. Finally, the labels $m(V)$ and $\chi(V)$ differ by at most one for all the nodes, so $\valsep(\CT') \asymp \valsep(\CT)$.

With the above in mind, let us denote $H(\rho, t) = \int_0^\infty h(\rho(B(t,\eps))d\eps$ and note that $\dual_h(\rho, X) = \min_{t \in X} H(\rho, t)$.  Consider the path $\CP$ from the root of the tree $\CT'$ to an arbitrary leaf $\{x\}$ and let $V_{0} =X, V_{1}, \cdots, V_{k} = \{x\}$ be all the nodes of this path. Note that by definition, we have that $\{x\} = V_{k} \subseteq V_{{k-1}} \subseteq \cdots \subseteq V_{0} = X$. 
   
For $\ell \in \{0\} \cup [k]$, let $m_\ell = m(V_\ell)$ and let $r_\ell = \frac12 \alpha^{m_\ell} \cdot \diam(X)$ denote an upper bound on the radius (one half of the diameter) of $V_{\ell}$ where we define $r_k=0$. Note that apart from the root, the radius of each $V_\ell$ is even smaller than the above bound by a factor of $1/2$, but we only need an upper bound for the rest of the proof. Then,
we have that 
\begin{align*}
\   H(\rho, x) &=\sum_{\ell=1}^k \int_{2r_\ell}^{2r_{\ell-1}} h(\rho(B(x,\eps)) d\eps \le \sum_{\ell=1}^k 2r_{\ell-1} h(\rho(V_{\ell})) \defeq \Gamma,
\end{align*}
where the second inequality follows since $h$ is decreasing and $V_{\ell} \subseteq B(x,2r_\ell) \subseteq B(t,\eps)$ for $\eps \in (2r_\ell, 2r_{\ell-1})$ for $\ell \in [k]$. Further, using sub-multiplicativity of $h$, 
\begin{align*}
\   \Gamma \le \sum_{\ell=1}^k 2r_{\ell-1} \left( h\left(\frac{\rho(V_{\ell})}{\rho(V_{{\ell-1})}}\right) + h\left(\rho(V_{{\ell-1}})\right)\right) \le \sum_{\ell=1}^k 2r_{\ell-1} h\left(\frac{\rho(V_{\ell})}{\rho(V_{{\ell-1}})}\right) +  \alpha \Gamma,
\end{align*}
since the $r_\ell$'s decrease at least by a factor of $\alpha$ as $\ell$ increases and $h(\rho(V_{0}))=h(1)=0$. We remark that the ratios $\rho(V_\ell)/\rho(V_{\ell-1})$ may be greater than $1$ but since we extended $h$ to the positive real line while maintaining sub-multiplicativity, the above bound remains valid. 

Rearranging the expression above, we get that
\begin{equation}\label{eqn:sep1}
 \ H(\rho,x) \le \Gamma \le \frac2{1-\alpha} \sum_{\ell=1}^k r_{\ell-1} h\left(\frac{\rho(V_{\ell})}{\rho(V_{{\ell-1}})}\right).
\end{equation}

Let $V_{i_1}, \cdots, V_{i_{q-1}}$ be the sequence of nodes of the path $\CP$, arranged from root to leaf, that were added during the execution of step (c-\textrm{ii.}). Defining $i_0=0$ and $i_q=k$, we shall prove that for every $\kappa \in \{0\} \cup [q]$, the contribution of the entire chunk of nodes between $V_{i_\kappa}$ and $V_{i_{\kappa+1}}$ is accounted by the parent of the last node. In particular, the following holds
\begin{equation}\label{eqn:sep2}
\ \sum_{\ell=i_\kappa}^{i_{\kappa+1}} r_{\ell-1} h\left(\frac{\rho(V_{\ell})}{\rho(V_{{\ell-1}})}\right)  \le C(\alpha) \cdot r(W) \cdot h \left(\frac1{\deg(W)}\right) + C(\alpha)  \cdot r_{i_{\kappa} } \cdot h\left(\frac12\right),
\end{equation}
where $W$ is the parent node of $V_{i_{\kappa+1}}$ with the corresponding radius $r(W) = \alpha^{-2}  r_{i_{\kappa+1}}$, and $C(\alpha) \le 40/(3\alpha^2)$ is a constant. Note that for the last chunk of path between $V_{i_{q-1}}$ and $V_{i_q} = \{x\}$ there might not be any nodes that were added during the execution of step (c-\textrm{ii.}), in which case we define $r(W)=0$ in the above expression.

Plugging the bound  given by \eqref{eqn:sep2} in \eqref{eqn:sep1},  it follows that 
$$ \dual_h(\rho, X) \le {4C(\alpha)}  \cdot \valsep(\CT') + {4C(\alpha)} \cdot \diam(X) \cdot h(1/2),$$ 
as $x$ was an arbitrary leaf of $\CT'$ and $\alpha\le 1/10$. By our assumption that $\dual_h(\rho, X) \ge 60\alpha^{-2} \cdot \diam(X) \cdot h(1/2)$, the above implies that $\dual_h(\rho, X) \lesssim \valsep(\CT')$ giving us statement of the lemma.

To finish we now prove \eqref{eqn:sep2}. Fix a value of $\kappa \in \{0\} \cup [q]$, and for convenience, let us relabel the nodes so that $V_{i_{\kappa}} := U_{0}, U_{1}, \cdots, U_{j}, U_{j+1} := V_{i_{\kappa+1}}$ is the path from $V_{i_{\kappa}}$ to $V_{i_{\kappa+1}}$. Also, by a slight abuse of notation, let us write $r_0, \cdots, r_{j+1}$ to denote the corresponding sequence $r_{i_\kappa}$'s of radii of the nodes $V_{i_\kappa}$ to $V_{i_{\kappa+1}}$.  Then, we shall show that the sum of the contributions of all but the last node in the left hand side of \eqref{eqn:sep2} can be charged to the last node:
\begin{claim}\label{clm:sep1}
 $\displaystyle \sum_{\ell=1}^{j} r_{\ell-1} h\left(\frac{\rho(U_{\ell})}{\rho(U_{{\ell-1}})}\right) \le \frac{\alpha}{4-\alpha} \cdot r_{j} \cdot h\left(\frac{\rho(U_{j+1})}{\rho(U_j)}\right).$
\end{claim}

Second, we show that the contribution of the last node is bounded by the right hand side of \eqref{eqn:sep2}.
 \begin{claim}\label{clm:sep2}
 $\displaystyle r_j h\left(\frac{\rho(U_{j+1})}{\rho(U_j)}\right) \le 10\alpha^{-2} \cdot r_{j} \cdot \left(h\left( \frac1{\deg(U_{j})}\right) + h\left(\frac12\right)\right).$
\end{claim}

Combining the above two claims yields \eqref{eqn:sep2} with $C(\alpha) = \frac{4}{4-\alpha} \cdot 10\alpha^{-2} \le 40/(3\alpha^2)$. This completes the proof of \lref{lem:sep} assuming the claims which we prove next.

\begin{proof}[Proof of \clmref{clm:sep1}]
	For $S \subseteq X$, let us define a potential
 \[\Psi(r, S) = \min_{x \in S} h\left( \frac{\rho(B(x, r) )}{\rho(S)}\right).\]
	We will prove by induction that for every $p \in \{0\} \cup [j]$, the partial sum
	\begin{equation}\label{eqn:clm1main}
	\ \sum_{\ell=1}^{p} r_{\ell-1} h\left(\frac{\rho(U_{\ell})}{\rho(U_{{\ell-1}})}\right) \le \left( \sum_{i=1}^p \left(\frac\alpha4\right)^i\right) \cdot r_{p} \cdot \Psi(\alpha r_p, U_p). 
	\end{equation}
Note that $U_{p+1} = B(t,\alpha r_p)$ for some $t \in U_p$. Since $h$ is decreasing, this implies that
\begin{equation}\label{eqn:psi}
 \ \Psi(\alpha r_p, U_p) \le h\left(\frac{\rho(U_{p+1})}{\rho(U_p)}\right),
\end{equation}
so taking $p=j$ and noting that $\sum_{i=1}^\infty (\alpha/4)^i = \alpha/(4-\alpha)$, the inequality \eqref{eqn:clm1main} yields the claim.

Now we proceed with the induction. Note that the stronger hypothesis with the potential $\Psi$ is crucial to carry out the induction. The base case when $p=0$ is trivial as $\Psi$ is non-negative. For the inductive step, consider an arbitrary $p \in [j]$ for which we have

	\begin{align}\label{eqn:clm1}
		\ \sum_{\ell=1}^{p} r_{\ell-1} h\left(\frac{\rho(U_{\ell})}{\rho(U_{{\ell-1}})}\right) &= \sum_{\ell=1}^{p-1} r_{\ell-1} h\left(\frac{\rho(U_{\ell})}{\rho(U_{{\ell-1}})}\right) + r_{p-1} h\left(\frac{\rho(U_p)}{\rho(U_{p-1})}\right)  \notag \\
		\ &\le \left( \sum_{i=1}^{p-1} \left(\frac\alpha4\right)^i\right) \cdot r_{p-1} \cdot \Psi(\alpha r_{p-1}, U_{p-1}) + r_{p-1} h\left(\frac{\rho(U_{p})}{\rho(U_{p-1})}\right) \notag\\
		\ &\le \left( \sum_{i=0}^{p-1} \left(\frac\alpha4\right)^i\right) \cdot r_{p-1} \cdot h\left(\frac{\rho(U_{p})}{\rho(U_{p-1})}\right),
\end{align}
where the first inequality follows from the inductive claim and the second from \eqref{eqn:psi}.
		
Now by construction, $h\left(\frac{\rho(U_{p})}{\rho(U_{p-1})}\right) \le \frac{\alpha^2}4 \cdot \left(\min_B h\left(\frac{\rho(B)}{\rho(U_{p})}\right)\right)$ where $B$ is a ball of radius $\alpha^2 r_{p-1} = \alpha r_p$ centered at some point in $U_p$. Thus, we obtain
\begin{align*}
		\  \eqref{eqn:clm1}  &\le  \left( \sum_{i=0}^{p-1} \left(\frac\alpha4\right)^i\right) \cdot r_{p-1} \cdot\frac{\alpha^2}4  \cdot \Psi(\alpha r_{p}, U_{p}) = \left( \sum_{i=1}^p \left(\frac\alpha4\right)^i\right) \cdot r_{p} \cdot \Psi(\alpha r_{p}, U_{p}),
	\end{align*}
since $r_p = \alpha r_{p-1}$. This gives us the inductive statement and completes the proof.
\end{proof}

\begin{proof}[Proof of \clmref{clm:sep2}]
	Let $t_1, \ldots, t_{i_*}$ be the centers chosen while processing $U_j$ and for $i \in [i_*]$, let $A_i = B(t_i, \alpha r_j) \cap S_{i-1}$ where $S_{i-1} = A_1 \cup \ldots \cup A_{i-1}$. By construction, $U_{j+1} = B_i$ for some $i \in L$ where $B_i = B(t_i, \alpha^2 r_j )$ and $L$ is the subset of $M$ as chosen in step (b). Moreover, by sub-multiplicativity we also have that 
	$$ r_j h\left(\frac{\rho(U_{j+1})}{\rho(U_j)} \right) \le r_j h\left(\frac{\rho(B_i)}{\rho(A_i)} \right) +r_j  h\left(\frac{\rho(A_i)}{\rho(U_j)} \right) \le (1+4\alpha^{-2}) r_j h\left(\frac{\rho(A_i)}{\rho(U_j)} \right),$$
where the second inequality holds since this node was added in step (c-\textrm{ii}.), so the condition in step (c-\textrm{i.}) is not satisfied.

Furthermore, for every $i \in L$, we have that $\rho(A_i)/\rho(U_j) \ge \frac{6}{\pi^2} \cdot \frac{1}{i^2} \ge \frac{6}{\pi^2} \cdot \frac{1}{|L|^2} \ge \frac12 \cdot \frac{1}{\deg(U_{j})^2}$, since $U_j$ has $|L|$ children. Since $\alpha = 1/10$, using this along with the fact  that  $h$ is decreasing and sub-multiplicative, we get that 
\begin{align*}
 \ r_j h\left(\frac{\rho(U_{j+1})}{\rho(U_j)} \right) &\le (1+4\alpha^{-2}) \cdot r_j  \cdot h\left(\frac{\rho(A_i)}{\rho(U_j)} \right)  \\
 \ & \le 10\alpha^{-2} \cdot r_{j} \left(h\left(\frac1{\deg(U_{j})}\right) + h\left(\frac12\right)\right).  \qedhere
 \end{align*}
\end{proof}
\end{proof}

\section{Applications}
\subsection{Cover time of graphs}

Consider the simple random walk on a finite connected graph $G = (V,E)$ started at a vertex $v$. Writing $\taucov$ to be the first time that every vertex of $G$ has been visited and denoting by $\BE[\taucov(v)]$ its expectation when the random walk is started at a vertex $v \in V$, the cover time of $G$ is defined to be the following fundamental statistic of the random walk,
\begin{align*}
\tcov(G)=\max_{v\in V}\BE[\taucov(v)].
\end{align*}

Ding, Lee and Peres \cite{DLP12} showed that the expected suprema of the
Gaussian free field, which is a natural Gaussian process associated with a
graph, characterizes the cover time up to constant factors. In particular, they
showed that $\tcov(G) \asymp |E| \cdot (\BE \max_{v \in V} \eta_v)^2$ where
$\{\eta_v\}_{v \in V}$ is the Gaussian free field. As the expected suprema is
equivalent to the $\primal_2$ functional on the corresponding metric space, to
give a constant factor approximation to cover time, it suffices to compute the corresponding $\primal_2$ functional. \cite{DLP12} gave a deterministic algorithm  to construct optimal packing trees and obtained a constant factor approximation for cover time. Subsequently, Meka \cite{M12} provided a deterministic PTAS to compute the suprema of a Gaussian process using different techniques --- this  gives an alternative deterministic way to compute the cover time up to constant factors. 

A third way to compute the value of the cover time is by solving the majorizing measures convex program. However as already mentioned, the deterministic algorithm of \cite{DLP12} to construct a packing tree is faster than this if one uses an off-the-shelf convex optimization algorithm. But it is quite conceivable that in our setting, where we only need a crude additive approximation to the convex program, improvements can be made; even near-linear time in the size of the input (which is $O(n^2)$) does not seem out of the question.

\subsection{Gordon's Theorem deterministically}
\label{sec:gordon}

The well-known Johnson-Lindenstrauss Lemma says that for any finite set $X \subset \BR^n$, there is a projection of dimension $m = O({\log |X|})$ such that all norms are approximately preserved. Gordon's theorem generalizes this to the setting of an arbitrary (even infinite) set $X \subset \BR^n$.

\begin{theorem}[Gordon~\cite{Gordon88}]\label{thm:gordon}
Let $X \subset \RR^n$ be a set such that $\|x\|_2 = 1$ for all $x \in X$, and let $w(X) = \E_{G} \sup_{x \in X} \langle G, x \rangle$ denote the Gaussian width of $X$, where $G$ is standard normal in $\BR^n$. Then for some $m=O((w(X)^2 + 1)/\epsilon^2)$, there is a projection $\Pi \in \RR^{m \times n}$ so that
\begin{equation}\label{eq:gordproj}
    \|\Pi x\|_2 \in [1-\epsilon, 1+\epsilon] \qquad \forall x \in X.
\end{equation}
\end{theorem}

Existing constructions for the claimed projection $\Pi$ in Gordon's Theorem are randomized; indeed, simply choosing i.i.d.\ Gaussians of variance $1/m$ for each entry of $\Pi$ suffices. Let us denote this distribution on matrices by $\mathcal{G}$ for the remainder of this section.

Klartag and Mendelson~\cite{KlartagMendelson05} gave a proof of Gordon's theorem exploiting majorizing measures.
We will follow a more recent work of Oymak, Recht, Soltanolkotabi~\cite{ORS18}, who use majorizing measures to prove a substantial generalization of Gordon's theorem.
(See also~\cite{NelsonNotes} for an exposition more directly applicable to Gordon's theorem.) 
Their procedure starts from an \emph{admissible net} for the Gaussian process defined on $(X, \|\cdot\|_2)$. 
This is a variant of the labelled net construction, and can also be thought of as a special kind of chaining tree.
\begin{definition}
    Given a metric space $(X,d)$, a sequence $(\CA_0, \CA_1, \ldots, \CA_k)$ of partitions of $X$, where
    $\CA_i$ is a refinement of $\CA_{i-1}$ for each $i \in [k]$, is called  an {admissible net} if
    $|\CA_i| \leq 2^{2^i}$ for each $i$. 
    
    The \emph{value} of an admissible net $\CA \defeq (\CA_i)_{i=1}^k$ is defined to be
    \[ 
        \val_2(\CA) \defeq \max_{x \in X} \sum_{i=0}^k g(2^{-2^i})\diam(\CA_i(x)), 
    \] 
    where  $\CA_i(x)$ is the partition element that contains $x$ and $g$ is the Gaussian functional.
\end{definition}

The above definition is tailored for the Gaussian functional $g$ and we remark that one can convert a labelled net to an admissible net in $O(n^2)$ time, preserving the value up to constant factors, at least for the case of the Gaussian functional $g$. The proof is similar to how a labelled net can be converted to a chaining tree and can be found in \appref{sec:proof_admissible_nets}. Thus, we are able to (deterministically) construct an admissible net $\CA$ with $\val_2(\CA)
\asymp \primal_2(X)$.

Based on the admissible net $\CA$, Oymak et al.~\cite[Lemma 4.6]{ORS18}
(cf.~\cite[Lemma 1]{NelsonNotes}) define a set of events $\mathcal{E}$ such that: 
\begin{itemize}
\item any projection matrix that satisfies all the events necessarily satisfies \eqref{eq:gordproj}, and 
\item $\sum_{E \in \mathcal{E}} \Pr_{\Pi \sim \mathcal{G}}(\overline{E}) < 1/2$.
\end{itemize}
Then by the union bound, $\Pi$ drawn from $\mathcal{G}$ satisfies all events with probability at least $1/2$, proving Gordon's theorem.

Conveniently, the events they use can all be phrased in terms of approximately preserving the norm of some vector. That is, each $E \in \mathcal{E}$ is of the form
\begin{equation}\label{eq:gorddistort}
\bigl| \|\Pi v_E\|_2^2 - 1 \bigr| \leq \epsilon_E
\end{equation}
for some unit vector $v_E$ and distortion bound $\epsilon_E > 0$.\footnote{The events phrased in terms of inner products in \cite{ORS18} and \cite{NelsonNotes} are also implied by the norm events using the parallelogram law.
The bound on the operator norm of $\Pi$ required by \cite{NelsonNotes} can be replaced by $|X|$ norm events, since it is only used to bound $\|\Pi(x - a_x)\|$ for all $x \in X$; here, $a_x = \argmin_{a \in A} \|x - a\|$, with $A$ being some set which intersects each part of $\mathcal{A}_k$ precisely once.}

We can derandomize this as follows.
First, with the results of this paper we can efficiently and deterministically construct a good admissible net $\CA$.
From this, we can compute the events $\mathcal{E}$, in particular the vectors $v_E$ and distortion bounds $\epsilon_E$.
We then apply a deterministic version of Johnson-Lindenstrauss to find a $\hat{\Pi}$ that preserves the norm of $v_E$ up to error $\epsilon_E$ for each $E \in \mathcal{E}$, i.e., satisfies \eqref{eq:gorddistort}.

Some care is required, since a ``lopsided'' form of Johnson-Lindenstrauss is needed: different vectors have different distortion bounds.
But we can plug things fairly directly into the pessimistic estimator framework of Dadush, Guzman and Olver~\cite{DGO18}. The following theorem can be easily extracted from their paper:
\begin{theorem}
Suppose we are given a finite set $V$ of unit vectors in $\RR^n$, bounds $\epsilon_v > 0$ for each $v \in V$, and $\lambda > 0$, $m \in \Z_+$ such that
\begin{equation}\label{eq:phibound}
\sum_{v \in V} [\Phi_v^+ + \Phi_v^-] < 1,
\end{equation}
where
\begin{align*}
\Phi_v^{\pm} = e^{-\lambda \epsilon_v}\cdot \E_{\Pi \sim \mathcal{G}} \Bigl( e^{\pm \lambda (\|\Pi v\|_2^2 - 1)}\Bigr).
\end{align*}
Then in time $O(mn^2)$, one can deterministically construct a projection $\tilde{\Pi} \in \{\pm1\}^{m \times n}$ so that
\[
    \bigl| \|\tilde{\Pi}v\|_2^2 - 1\bigr| \leq \epsilon_v \quad \text{ for each } v \in V.
\]
\end{theorem}
This theorem requires slightly more than bounds on $\sum_{E \in \mathcal{E}} \Pr_{\Pi \sim \mathcal{G}}(E)$.
The relation is that for any unit vector $v \in \RR^n$ and $\epsilon_v > 0$, 
\[ \Pr_{\Pi \sim \mathcal{G}}(\|\Pi v\|_2^2 - 1 > \epsilon_v) \leq \Phi^+_v \quad \text{and} \quad
    \Pr_{\Pi \sim \mathcal{G}}(\|\Pi v\|_2^2 - 1 < -\epsilon_v) \leq \Phi^-_v,
\]
via the usual Chernoff-Cram\'er approach.
But the bounds on $\Pr_{\Pi \sim \mathcal{G}}(E)$ are anyway best obtained through concentration inequalities (most conveniently, the Hanson-Wright inequality, see~\cite{KN14}), and
it is not hard to check that the condition \eqref{eq:phibound} of the above theorem can be satisfied for the set of events given by \cite{ORS18}. We omit the details.

\bibliographystyle{alpha}
{\footnotesize
\bibliography{chaining}

\begin{thebibliography}{JLSW20}

\bibitem[Bed06]{Bednorz06}
Witold Bednorz.
\newblock A theorem on majorizing measures.
\newblock {\em The Annals of Probability}, 34(5):1771--1781, 2006.

\bibitem[Bed12]{B12}
Witold Bednorz.
\newblock The majorizing measure approach to sample boundedness.
\newblock {\em Colloquium Mathematicum}, 139, 11 2012.

\bibitem[DGO18]{DGO18}
Daniel Dadush, Cristóbal Guzmán, and Neil Olver.
\newblock Fast, deterministic and sparse dimensionality reduction.
\newblock In {\em Proceedings of the 2018 Annual ACM-SIAM Symposium on Discrete
  Algorithms (SODA)}, pages 1330--1344, 2018.

\bibitem[DLP12]{DLP12}
Jian Ding, James~R Lee, and Yuval Peres.
\newblock Cover times, blanket times, and majorizing measures.
\newblock {\em Annals of mathematics}, 175(3):1409--1471, 2012.

\bibitem[Dud67]{Dudley67}
Richard~M Dudley.
\newblock The sizes of compact subsets of {H}ilbert space and continuity of
  {G}aussian processes.
\newblock {\em Journal of Functional Analysis}, 1(3):290--330, 1967.

\bibitem[Fer75]{Fernique75}
Xavier Fernique.
\newblock Regularit{\'e} des trajectoires des fonctions al{\'e}atoires
  gaussiennes.
\newblock In {\em Ecole d'Et{\'e} de Probabilit{\'e}s de Saint-Flour IV--1974},
  pages 1--96. Springer, 1975.

\bibitem[Fuj09]{Fujishige2009}
Satoru Fujishige.
\newblock {\em Theory of Principal Partitions Revisited}, pages 127--162.
\newblock Springer Berlin Heidelberg, Berlin, Heidelberg, 2009.

\bibitem[Gor88]{Gordon88}
Yehoram Gordon.
\newblock On {M}ilman's inequality and random subspaces which escape through a
  mesh in {R}n.
\newblock In {\em Geometric aspects of functional analysis}, pages 84--106.
  Springer, 1988.

\bibitem[GZ03]{GZ03}
Olivier Gu{\'e}don and Artem Zvavitch.
\newblock {\em Supremum of a Process in Terms of Trees}, pages 136--147.
\newblock Springer Berlin Heidelberg, Berlin, Heidelberg, 2003.

\bibitem[JLSW20]{JLSW20}
Haotian Jiang, Yin~Tat Lee, Zhao Song, and Sam~Chiu{-}wai Wong.
\newblock An improved cutting plane method for convex optimization,
  convex-concave games, and its applications.
\newblock In {\em Proccedings of the 52nd Annual {ACM} {SIGACT} Symposium on
  Theory of Computing, {STOC} 2020}, pages 944--953. {ACM}, 2020.

\bibitem[KM05]{KlartagMendelson05}
B.~Klartag and S.~Mendelson.
\newblock Empirical processes and random projections.
\newblock {\em Journal of Functional Analysis}, 225(1):229--245, Aug 2005.

\bibitem[KN14]{KN14}
Daniel~M. Kane and Jelani Nelson.
\newblock Sparser {J}ohnson-{L}indenstrauss {T}ransforms.
\newblock {\em Journal of the ACM}, 61(1), January 2014.

\bibitem[LT91]{LT91}
Michel Ledoux and Michel Talagrand.
\newblock {\em Probability in Banach Spaces: Isoperimetry and Processes},
  volume~23.
\newblock Springer Science \& Business Media, 1991.

\bibitem[Mek15]{M12}
Raghu Meka.
\newblock A polynomial time approximation scheme for computing the supremum of
  {G}aussian processes.
\newblock {\em Annals of Applied Probability}, 25(2):465--476, 04 2015.

\bibitem[Mil71]{Milman71}
Vitali~D Milman.
\newblock A new proof of {A}. {D}voretzky's theorem on cross-sections of convex
  bodies.
\newblock {\em Funkcional. Anal. i Prilozen}, 5:28--37, 1971.

\bibitem[MN11]{MN11}
Manor Mendel and Assaf Naor.
\newblock Ultrametric subsets with large {H}ausdorff dimension.
\newblock {\em Inventiones mathematicae}, 192, 06 2011.

\bibitem[MN13]{MN13}
Manor Mendel and Assaf Naor.
\newblock Ultrametric skeletons.
\newblock {\em Proceedings of the National Academy of Sciences},
  110(48):19256--19262, 2013.

\bibitem[Nel15]{NelsonNotes}
J.~Nelson.
\newblock Lecture notes for ``{A}lgorithms for {B}ig {D}ata'', {L}ecture 13.
\newblock
  \url{http://people.seas.harvard.edu/~minilek/cs229r/fall15/lec/lec13.pdf},
  2015.

\bibitem[ORS18]{ORS18}
Samet Oymak, Benjamin Recht, and Mahdi Soltanolkotabi.
\newblock Isometric sketching of any set via the {Restricted Isometry
  Property}.
\newblock {\em Information and Inference: A Journal of the IMA}, 7(4):707--726,
  2018.

\bibitem[Sle62]{Slepian62}
David Slepian.
\newblock The one-sided barrier problem for {G}aussian noise.
\newblock {\em Bell System Technical Journal}, 41(2):463--501, 1962.

\bibitem[Sud71]{Sudakov71}
Vladimir~Nikolaevich Sudakov.
\newblock Gaussian random processes and measures of solid angles in {H}ilbert
  space.
\newblock In {\em Doklady Akademii Nauk}, volume 197, pages 43--45. Russian
  Academy of Sciences, 1971.

\bibitem[Tal87]{Talagrand87}
Michel Talagrand.
\newblock Regularity of {G}aussian processes.
\newblock {\em Acta mathematica}, 159:99--149, 1987.

\bibitem[Tal90]{Talagrand90}
Michel Talagrand.
\newblock Sample boundedness of stochastic processes under increment
  conditions.
\newblock {\em The Annals of Probability}, pages 1--49, 1990.

\bibitem[Tal92]{Talagrand92}
Michel Talagrand.
\newblock A simple proof of the majorizing measure theorem.
\newblock {\em Geometric \& Functional Analysis GAFA}, 2(1):118--125, 1992.

\bibitem[Tal96]{T96}
Michel Talagrand.
\newblock Majorizing measures: the generic chaining.
\newblock {\em Annals of Probability}, 24(3):1049--1103, 07 1996.

\bibitem[Tal01]{T01}
Michel Talagrand.
\newblock Majorizing measures without measures.
\newblock {\em Annals of Probability}, 29(1):411--417, 02 2001.

\bibitem[vH16]{vH16}
Ramon van Handel.
\newblock Probability in {H}igh {D}imensions.
\newblock Lecture Notes. Princeton University, 2016.

\end{thebibliography}
}

\appendix

\section{Proofs Omitted from the Introduction and Preliminaries}\label{sec:app-prelims}

\subsection{Constructing Majorizing Measures from Chaining Trees.}

\begin{lemma}\label{lem:cctbound}
Given a chaining tree $\CCT$, one can construct a measure $\mu$ whose
value in \eqref{eq:maj-meas-ub} is at most $3 \cdot \val_h(\CCT)$.
\end{lemma}
\begin{proof}
    We recall the construction of the measure $\mu$ sketched in the introduction.
Without loss of generality assume that the
edge probabilities satisfy $\sum_{e \in E[\CCT]} p_e = 1/2$. Letting $w \in X$
denote the root of $\CCT$, we define the $\mu$ as follows: firstly, we put mass
$1/2$ on $w$, and for every edge $e = \{u,v\} \in E[\CCT]$ directed towards the
root we put mass $p_e$ on $u$. One can directly verify that $\mu$ is indeed a
probability measure. Let us now relate the value of $\mu$ as
in~\eqref{eq:maj-meas-ub} to $\val_h(\CCT)$.  

Take $x \in X$. Let $x = v_1,v_2,\dots,v_k = w$ denote the sequence vertices of
$\CCT$ followed on the unique path from $x$ to $w$. For $i \in [k]$, define $r_i
= \sum_{j=1}^{i-1} d(v_j,v_{j+1})$, where $r_1 = 0$ by convention. Note that by
the triangle inequality, $d(x,v_i) \leq r_i$. In particular, for $i \in [k]$ and
$r \geq r_i$, we have $\mu(B(x,r)) \geq \mu(\{v_i\}) = p_{v_i,v_{i+1}}$, for $i
< k$, and $\mu(B(x,r)) \geq \mu(\{w\}) \geq 1/2$ for $i=k$. From here, we have that 
\begin{align}
\int_0^\infty h(\mu(B(x,r))) dr 
&= \left(\sum_{i=1}^{k-1} \int_{r_i}^{r_{i+1}} h(\mu(B(x,r))) dr\right) +
\int_{r_k}^\infty h(\mu(B(x,r))) dr \nonumber \\  
&\leq \sum_{i=1}^{k-1} d(v_i,v_{i+1}) h(p_{v_i,v_{i+1}}) + \diam(X) \cdot h(1/2) \nonumber \\
&\leq \sum_{e \in \CP_x} l_e + \diam(X) \cdot h(1/2) \leq 3\cdot \val_h(\CCT), \label{eq:tree-to-meas}
\end{align}
where the last inequality follows from~\eqref{eq:trivial-lb}. 
\end{proof}

\subsection{Packing Trees Gives Lower Bounds on $\mathbf{\gamma}_h(X)$}

\weakdual*

\begin{proof}

First of all, since nodes with a single child do not contribute to the value of the packing tree by definition, we may assume without loss of generality, that the packing tree does not have such nodes  (repeatedly replacing such a node by its sole child until necessary). Now, let $\rho$ be an arbitrary probability measure  on $X$.  For a node $A$ of $\CT$, define $S(A):=\bigcup_{x\in A}B(x,\frac12 \alpha^{\chi(A)} \cdot \diam(X))$. Consider an arbitrary node $V$ in the tree $\CT$ with children $C_1,\ldots,C_L$. Because of the properties of the packing tree, the sets $S(C_i)$ are pairwise disjoint. Consequently there is an $i \in [L]$ such that $\rho(S(C_i)) \leq 1/L$ because $\sum_{i=1}^L \rho(S(C_i)) \le 1$. This implies the existence of a path $V_0:=X,\ldots, V_k = \{x\}$ down the tree with $\rho(S(V_{j+1})) \leq 1/\deg(V_j)$ for all $j \in \{0\} \cup [k-1]$.

 Furthermore, since $\alpha \le 1/10$ and each node has more than one child, the definition of an $\alpha$-packing tree implies that the $\chi$ values increase by at least one on moving down a path. Let us write $r_j = \frac12 \alpha^{\chi(V_j)} \cdot \diam(X)$ for each $r \in \{0\} \cup [k]$ where $r_k=0$ by convention. Then, by definition, we have $x \in V_j$ for all $j \in \{0\} \cup [k]$, so also $S(V_j)\supseteq B(x,	r_j)$. This implies the following inequality,
\begin{align*}
\primal_h(\rho, X)&=\sup_{t \in X}\int_0^\infty h(\rho(B(t, \eps)))d\eps \geq\int_0^\infty h(\rho(B(x, \eps)))d\eps \ge \sum_{j=0}^k \int_{r_j}^{r_{j+1}} h(\rho(S(V_j))) d\eps
\end{align*}
where the last inequality holds since $h$ is decreasing and for $\eps \in (r_{j+1}, r_j)$ we have that $B(x, \eps) \subseteq B(x, r_j) \subseteq S(V_j)$.

From the above, we have that 
\begin{align*}
\ \primal_h(\rho, X) &\ge   \sum_{j=0}^k (r_j - r_{j+1}) \cdot  h\left(\frac{1}{\deg(V_j)}\right) \ge \frac12 \sum_{j=0}^k \alpha^{\chi(V_j)} (1-\alpha) \cdot h\left(\frac{1}{\deg(V_j)}\right)\ge \frac{1-\alpha}2 \cdot \valsep(\CT),
\end{align*}
which implies the statement of the lemma as $\rho$ was arbitrary.

\end{proof}

\subsection{Properties of Chaining Functionals}

We first show that chaining functionals of log-concave type satisfy sub-multiplicativity. Recall that we have a log-concave density  $f$ on the positive real line with the normalization $f(0)=1$ and $F(s) = \int_s^\infty f(s)ds$ is the corresponding complementary cumulative distribution function. Furthermore, $h(p) = F^{-1}(p)$ for $p \in (0,1]$  which also implies that $h'(1) = -1/f(0) = -1$. 

\submultiplicative*
\begin{proof}
Let us define $\phi: [0,\infty) \to [0, \infty)$ as $\phi(t):=h(e^{-t})$. We will show that $\phi$ is concave on its domain, which implies that $h$ is sub-additive as 
 \[ h(ab)=\phi\left(\log \frac1{ab}\right)\le \frac12 \phi\left(2\log \frac1a\right) + \frac12 \phi\left(2\log \frac1b\right) \leq \phi\left(\log \frac1a\right) + \phi\left(\log \frac1b\right) = h(a)+h(b),\]
 where both inequalities follows from concavity and $\phi(0)=0$.

To see that $\phi$ is concave, we show that $\phi'(t)=-e^{-t}h'(e^{-t}) = \frac{e^{-t}}{f(F^{-1}(e^{-t}))}$ is a decreasing function of $t$. Substituting $x=F^{-1}(e^{-t})$, and noting that $f$ is decreasing, it suffices to show that $\frac{F(x)}{f(x)}$ is decreasing for $x$ on the positive real line. Taking arbitrary $0 \le x_1 \le x_2$, we have that
\[ \frac{F(x_1)}{f(x_1)} = \int_0^\infty \frac{f(x_1+t)}{f(x_1)} dt \ge \int_0^\infty \frac{f(x_2+t)}{f(x_2)} dt =  \frac{F(x_2)}{f(x_2)},\]
where the inequality follows from the following elementary property of non-negative log-concave functions: for any four points $a\le b \le c \le d$, we have that $f(b)f(c) \ge f(a)f(d)$ which in the above scenario implies that $\frac{f(x_1+t)}{f(x_1)} \ge \frac{f(x_2+t)}{f(x_2)}$. This completes the proof of the proposition.

\end{proof}

Using the previous lemma, one can show the following property of log-concave chaining functionals. Recall our normalization that $h'(1)=-1$.

\derivative*
\begin{proof}
	As $h$ is decreasing, $h'(a) \le 0$ for $a \in (0,1]$. Therefore, for any $a$,
	\begin{align}\label{eqn:hp}
	\ ah'(a) = \lim_{\epsilon \to 0^+} \frac{a (h(a) - h(a(1-\epsilon)))}{a - a(1-\epsilon)} &\ge \lim_{\epsilon \to 0^+} \frac{h(a) - h(a) - h(1-\epsilon)}{1 - (1-\epsilon)} \notag \\
	\ & = \lim_{\epsilon \to 0^+} \frac{h(1) - h(1-\epsilon)}{1 - (1-\epsilon)} = h'(1) = -1, 
	\end{align}
	where the first inequality follows from sub-multiplicativity and the last equality holds since $h(1)=0$. The first statement in the proposition follows.
	
	To see the second statement, one can observe that as $h$ is decreasing, \eqref{eqn:hp} implies the following differential inequality : $h'(a) \ge -\frac{1}{a}$ for every $a \in (0,1]$. Together with the boundary condition that $h(1)=0$, this implies that $h(a) \le  \log(1/a)$.
\end{proof}

\section{Algorithm for Constructing Measures in \thmref{thm:constructive}}\label{sec:alg}

In this section, we prove the second part of \thmref{thm:constructive} and show that one can construct almost optimal solutions $\mu^*$ and $\nu^*$ for the primal and the entropic dual algorithmically. For this, we will use the algorithm for convex-concave optimization given by the following theorem. We recall our normalization that $h'(1)=-1$.

\newcommand{\infnorm}[1]{\|#1\|_{\infty}}

\begin{theorem}[\cite{JLSW20}]\label{thm:mirror}
	Let $\CX \subseteq B(0,R) \subseteq \BR^n$ and $\CY \subseteq B(0,R) \subseteq \BR^n$ be compact convex sets such that they both contain a Euclidean ball of radius $r$. Let $\varphi(x,y)$ be an $L$-Lipschitz function with respect to the Euclidean norm that is convex in $x$ and concave in $y$. Then, for any $0<\epsilon\le 1/2$, one can find $(x^*,y^*)$ such that
	\[ \max_{y \in \CY} \varphi(x^*,y) - \min_{x \in \CX} \varphi(x,y^*)  \le \epsilon L r,\]
deterministically in time 
\[ O\left(n^{1+\omega} \log\left(\frac{n}{\epsilon} \cdot \frac{R}{r} \right)  + n \log\left(\frac{n}{\epsilon} \cdot \frac{R}{r} \right) T\right),\]
where $T$ is the time required to compute a subgradient $\nabla \varphi = \nabla_x \varphi(x,y) - \nabla_y \varphi(x,y)$ and $\omega \le 2.373$ is the matrix multiplication exponent. 
\end{theorem}

We remark that the one can replace $\omega$ in the above theorem with the rectangular matrix multiplication exponent where the matrices being multiplied are of dimensions $n\times n$ and $n \times n^{2/3+\varepsilon}$ for $\varepsilon > 0$.

In our context, the optimization problem we want to solve is given by
\begin{equation}\label{eqn:phi}
    \ \min_\mu \max_\nu \phi(\mu, \nu) ~~\text{ where }~~ \varphi(\mu, \nu) \triangleq \int_X \int_0^\infty h(\mu(B(t,\eps)))d\eps d \nu(x),
\end{equation}
which is convex in $\mu$ (as $h$ is convex on $(0,1]$) and concave in $\nu$. Here $\mu$ and $\nu$ lie in the standard simplex in $n$ dimensions. To avoid dealing with dimensionality issues, we observe that one can also optimize (for both $\mu$ and $\nu$) over the interior of the simplex without changing the value of the above optimization problem. 

To apply \thmref{thm:mirror}, we will take $\CX=\Delta^{1-1/e}_n$ and $\CY=\Delta^0_n$, where for $0 \le \alpha \le 1$, we define
\[ \Delta^\alpha_n = \left\{ \mu \in [0,1]^n ~~\Bigg|~~  \sum_{i=1}^d \mu_i \le 1 \text{ and } \mu_i \ge \frac{\alpha}{n} \text{ for each } i \in [n]\right\}.\]

Note that $\Delta^0_n$ is the standard simplex including its interior, while $\Delta^\alpha_n$ is a truncated simplex including its interior. The truncated simplex is needed to get a bound on the Lipschitz constant of $\varphi$ and we next show that taking $\CX = \Delta^{1-1/e}_n$ instead of $\Delta^{0}_n$ does not cause too much error.

\begin{claim} \label{clm:trnc}
Let $\mu$ be any probability measure on $X$ and let $0 \le \alpha \le 1$. Let $p = (1-\alpha)\mu + \alpha \cdot \frac{1}{n}\cdot \ind $ denote the probability measure that is the convex combination of $\mu$ with the uniform measure on $X$. Then, the following holds for any $t \in X$, 
    \[ \int_0^\infty h(p(B(t,\eps)))d\eps \le \int_0^\infty h(\mu(B(t,\eps)))d\eps + \log(1/(1-\alpha)) \cdot \diam(X). \]
\end{claim}

The above shows that setting $\alpha=1-1/e$ only causes an additive error of $\diam(X))$. We will prove the above claim later; let us first introduce the other claims we need to finish the analysis of the algorithm. 

Next, we show that for any $\mu \in \Delta^{1-1/e}_n$, we can get a reasonable bound on the Lipschitz constant of $\varphi$.

\begin{claim} \label{clm:lip}
Let $\mu \in \CX$ and $\nu \in \CY$. Then, $\varphi(\mu,\nu)$ is $L$-Lipschitz with respect to the Euclidean norm with $L =O(n^{3/2} \cdot \diam(X))$.
Moreover, a sub-gradient $\nabla\varphi = \nabla_\mu \varphi - \nabla_\nu \varphi$ at any point $(\mu, \nu)$ can be computed in $O(n^2 \log n)$ time.
\end{claim}

Given \clmref{clm:trnc} and \clmref{clm:lip}, we can finish the analysis of the algorithm first.
First, note that both $\CX$ and $\CY$ are contained in the Euclidean ball of radius $R=1$ and secondly, both $\CX$ and $\CY$ contain a ball of radius $r=\Omega(1/n)$ in their interior --- one can see that for any $0 \le \alpha \le 1$, the convex set $\Delta^\alpha_n$ contains a ball of radius $\frac{1-\alpha}{2n}$ around the point $\frac{1+\alpha}{2n}\cdot \ind$ where $\ind$ is the all ones vector in $\BR^n$.

We set $\epsilon = c \frac{1}{\sqrt{n}}$ for a suitable constant $c$, so that $\epsilon L r =  \diam(X)$. Then, \thmref{thm:mirror} and \clmref{clm:lip} imply that deterministically in $O(n^{1+\omega}\log^2 (n))$ time, one can find measures $\mu^*$ and $\nu^*$ satisfying 
\[ \max_{\nu \in \CY} \varphi(\mu^*,\nu) - \min_{\mu \in \CX} \varphi(\mu,\nu^*)  \le  \diam(X).\]

Since the duality gap (the left hand side above) must be non-negative because of Sion's minimax Theorem, it follows that $\mu^*$ and $\nu^*$ are almost optimal solutions, up to an additive error of $\diam(X)$, while optimizing over $\CX$ and $\CY$. 

Finally, as \clmref{clm:trnc} implies that  optimizing over the truncated simplex does not cause too much error, the above solutions $\mu^*$ and $\nu^*$ are almost optimal solutions to the optimization problem in \eqref{eqn:phi}, up to an additional additive error of $2\diam(X)$.

To finish, we prove \clmref{clm:trnc} and \clmref{clm:lip}.

\begin{proof}[Proof of \clmref{clm:trnc}]
    As $h$ is decreasing, for any $a,b \in (0,1]$ satisfying $(1-\alpha) a + \alpha b \in (0,1]$  we have that
    \[ h((1-\alpha) a + \alpha  b) \le h((1-\alpha )a) \le h(a) + h(1-\alpha) \le h(a) + \log(1/(1-\alpha)),\]
    where the second inequality uses the sub-multiplicativity property of $h$ and the last inequality follows from \pref{prop:derivative}. It follows that 
    \begin{align*}
         \ \int_0^\infty h(p(B(t,\eps)))d\eps &\le \int_0^\infty \Big(h(\mu(B(t,\eps))) + \log(1/(1-\alpha)) \cdot \ind[{\eps \le \diam(X)}]\Big)d\eps,\\
         \ &= \int_0^\infty h(\mu(B(t,\eps)))d\eps + \log(1/(1-\alpha)) \cdot \diam(X). \qedhere
    \end{align*}
\end{proof}

\begin{proof}[Proof of \clmref{clm:lip}]
    First, we note that for any $\mu \in \CX,\nu \in \CY$, we have that
    \begin{align}\label{eqn:grad}
         \ \nabla_\mu \varphi &= \left(\int_0^X \int_0^\infty h'(\mu(B(x,\eps))) \cdot \ind[t \in B(x,\eps)] d\eps d\nu(x)\right)_{t \in X}  \\
         \ \nabla_\nu \varphi &= \left( \int_0^\infty h(\mu(B(t,\eps))) d\eps\right)_{t \in X}   
    \end{align}
    which are both vectors in $\BR^{n}$ by identifying the $n$ coordinates with the index set $X$. We bound the infinity norm of the above gradients separately.
    
    Since, $\mu(t) \ge \frac{1}{en}$ for every $t \in X$, using the fact that $h$ is decreasing together with \pref{prop:derivative}, we have that for every $t \in X$, the following holds
        $$\int_0^\infty h(\mu(B(t,\eps)) d\eps \le \int_0^{\diam(X)} h\left(\frac{1}{en}\right)d\eps \le \log (en)  \cdot \diam(X).$$ This gives us that
    \begin{equation}\label{eqn:nu}
           \infnorm{\nabla_\nu \varphi} \le \log (en)  \cdot \diam(X).
    \end{equation}

    To bound the infinity norm of $\nabla_\mu \varphi$, we note that $|h'(a)|\le 1/a$ for every $a \in (0,1]$ due to \pref{prop:derivative}. Then, since  $\mu(B(t,\eps)) \ge 1/(en)$ for every $t \in X$, we have that for every $t \in X$, it holds that
    $$\int_0^\infty |h'(\mu(B(t,\eps))| d\eps \le \int_0^{\diam(X)} en \cdot d\eps  \le en \cdot \diam(X),$$
    giving us that 
    \begin{equation}\label{eqn:mu}
         \infnorm{\nabla_\mu \varphi} \le     en\cdot \diam(X).
    \end{equation}
   
   Using \eqref{eqn:mu} and \eqref{eqn:nu}, we have that Euclidean norm of the subgradient $\nabla \varphi = \nabla_\mu \varphi - \nabla_\nu \varphi$ is at most $L = O(n^{3/2} \cdot \diam(X))$. This gives us that $\varphi$ is $L$-Lipschitz over $\CX \times \CY$.
   
   \paragraph{Running Time.} Next we show that the subgradient can be computed in time $O(n^2 \log n)$ at any point $(\mu, \nu)$. First note that we can sort all pairs of points in the index set $X$ in time $O(n^2 \log n)$ according to the pairwise distances. Once they are sorted, for a fixed $x \in X$, one can compute all distinct values of $\eps$ such that the neighborhoods $B(x,\eps)$ change --- this can be done in $O(\log n+n) = O(n)$ time for each $x \in X$ by binary search and a linear scan.
   
   Given any fixed $x \in X$, let $0=\eps_0<\eps_1(x) < \eps_2(x) < \cdots < \eps_k(x)$ denote the distinct values of $\eps$ where $k = k(x)$ depends on $x$. Then, for a fixed $t\in X$, note that the integrals in \eqref{eqn:grad} are given by
   \begin{align*}
         \ \int_0^\infty h(\mu(B(t,\eps))) d\eps &= \sum_{\kappa=0}^{k(t)} h(\mu(B(t,\eps_{\kappa-1}(t)) \cdot (\eps_{\kappa}-\eps_{\kappa-1}), \text{ and }, \\
         \ \int_X \int_0^\infty h'(\mu(B(x,\eps))) \cdot \ind[t \in B(x,\eps)] d\eps d\nu(x) &= \sum_{x \in X} \nu(x) \sum_{\kappa=k_*(t,x)+1}^{k(x)} h'(\mu(B(x,\eps_{\kappa-1}(x)) \cdot (\eps_{\kappa}-\eps_{\kappa-1}), 
   \end{align*}
 where $k_*(t,x)$ is the smallest index $k$ such that $t \in B(x,\eps_k(x))$ (we set $k_*(t,x)=\infty$ if no such index exists).
   
   From the above, by using a straightforward dynamic program, one can compute the above integrals in $O(n)$ time for a given $t \in X$. This implies that the subgradient $\nabla \varphi$ can be computed in time $O(n^2 \log n + n \cdot n) = O(n^2 \log n)$.
   
\end{proof}

\section{Proof of \propref{clm:match}}\label{sec:hall}

\begin{proof}
    As we will modify the graph $G$, let us write $\partial_H(x)$ to denote the set of edges of the graph $H$ incident on the vertex $x$ and similarly, let $N_H(S)$ denote the neighborhood of $S$ in the graph $H$. 

    To prove the proposition, we may assume without loss of generality that $\supp(\nu)=X_2$. To construct the sequences $S_i$'s and $\beta_i$'s, we iterate the following process where we initialize $i=1$ and define $G_0=G$, and $S_0 = \emptyset$, and $N_{G_0}(S_0)=\emptyset$. Let us set $\beta_0=0$. 
    \begin{itemize}
        \item  Remove the vertices $S_{i-1} \cup N_{G_{i-1}}(S_{i-1})$ and all edges incident on them from the graph $G_{i-1}$ to obtain the graph $G_i$. Choose  $\beta_i = \min\left\{ \frac{\nu(N_{G_i}(S'))}{\mu(S')} \mid \emptyset \neq S' \subseteq X_1 \setminus S_{i-1} \right\}$ and choose $S'_i$ to be a maximal set $S'$ that achieves this minimum.  We define $S_i = S'_i \cup S_{i-1}$. Keeping in line with the intuition presented, we note that $S'_i$ is the least matchable set in the graph $G_i$ and only a $\beta_i$ fraction of mass in $S'_i$ can be transported to $N_{G_i}(S'_i)$, although we will not need this fact in the proof. 
 \item  If $S_i = X_1$, we terminate the process and define $k=i$, otherwise we increment $i$ and repeat the above.
    \end{itemize}

	We next show that $0<\beta_1<\beta_2<\cdots<\beta_{k-1} < \beta_k$ which also implies that the process terminates as the sets $\emptyset = S_0 \subset S_1 \subset \cdots \subset S_k= X_1$ form a strictly increasing family.  The fact that $\beta_1 > 0$ follows since we assumed that $\supp(\nu)=X_2$. To argue that the sequence $(\beta_i)_{i}$ is strictly increasing, we note that for any $1 < i \le k$,  the following holds
    \begin{align*}
        \ \beta_{i-1} \mu(S_i \setminus S_{i-2}) < \nu(N_{G_{i-1}}(S_i \setminus S_{i-2})) &= \nu(N_{G_{i}}(S_i \setminus S_{i-1})) + \nu(N_{G_{i-1}}(S_{i-1} \setminus S_{i-2}))\\
        \                                        &=\beta_i \mu(S_i \setminus S_{i-1}) + \beta_{i-1}  \mu(S_{i-1} \setminus S_{i-2}),
    \end{align*}
    where the strict inequality holds due to the maximality of $S'_i$. Rearranging the above implies that $\beta_{i-1}  \mu(S_i \setminus S_{i-1})< \beta_i \mu(S_i \setminus S_{i-1}) \implies \beta_{i-1} < \beta_{i}$.

Finally, we argue that the sets $S_i$'s and the sequence $(\beta_i)_i$ satisfies the required properties. The first statement follows directly from the definition of the sets $S'_i$. To see the second statement, we note that for every $i \in [k]$ and $A \subseteq X_1 \setminus S_{i-1}$, it holds that $N_{G_i}(A) = N_G(A)\setminus N_G(S_{i-1})$ and so by definition of $\beta_i$, we have $\beta_i \mu(N_G(A)) \le \nu(N_G(A)\setminus N_G(S_{i-1}))$. This completes the proof.

\end{proof}

\section{Construction of Admissible nets from Labelled nets}
\label{sec:proof_admissible_nets}

Recall that $g$ denotes the Gaussian chaining functional and $\val_2$ denotes the value of labelled nets, admissible nets, etc. with respect to the Gaussian functional. For the following proof, we will use that the Gaussian functional satisfies that $g(p^2) \ge C g(p)$ for $p \in [0,1)$ where $C$ is a universal constant satisfying $1<C\le 2$, in addition to sub-multiplicativity. Note that for any log-concave functional $h$ we have that $h(p) \le h(p^2) \le 2h(p)$, so the above condition is slightly stronger.

\begin{lemma}
Given a labelled net $\LN$, there exists a deterministic algorithm that constructs an admissible net $\CA$ satisfying $\val_2(\CA)\lesssim \val_2(\LN)$ in $O(n^2)$ time.
\end{lemma}

\begin{proof}
We essentially repeat the construction given in \cite[Theorem 6.30]{vH16} which uses the approximation $\sqrt{\log(1/p)}$ for the Gaussian functional $g(p)$. For a node $V$ in the labelled net, let $V_0=X, V_1, \ldots, V_k = V$ be the path from the root to the node $V$. Then, define the following potential
\begin{align*}
\ \Psi(V)= \prod_{i=1}^k (2\sigma(V_i))^2,
\end{align*}
where the potential of the root is defined to be zero. The value of the potential always increases as one goes down the tree, so for $i \in \mathbb{N}$, let us define the cut $\CA_i$ in the tree given by the set of nodes $V$ where the value of the potential $\Psi(V) < 2^{2^i}$ but for any node $W$ below the cut, the value of the potential $\Psi(W) \ge 2^{2^i}$. Observe that  $(\CA_i)_{i\in \mathbb{N}}$ forms an increasing sequence of partitions and it can be computed in $O(n^2)$ time. We claim that $(\CA_i)_{i\in \mathbb{N}}$ forms a valid admissible net as $|\CA_i| \le 2^{2^i}$. This follows since
\begin{align*}
|\CA_i|&\leq \sum_{k=1}^\infty \left|\left\{\{\ell_1,\ldots, \ell_k \} \subseteq \N \mid 4^k\prod_{m=1}^k \ell_m^2<2^{2^i} \right\}\right|\\
&\leq 2^{2^i}\sum_{k=1}^\infty 4^{-k}\sum_{\{\ell_1,\ldots, \ell_k \} \subseteq \N}\prod_{m=1}^k \ell_m^{-2} =  2^{2^i} \sum_{k=1}^\infty \left(\frac{\pi^2}{24}\right)^k < 2^{2^i}.
\end{align*}

Next we argue that $\val_2(\CA)\lesssim \val_2(\LN)$. Consider any $x \in X$ and let $V_0=X, V_1(x), \ldots, V_{l(x)}(x)=\{x\}$ be the path from root to the leaf $x$ in the labelled net $\CL$. Then, suppressing the dependence on $x$ for brevity, we claim that
\begin{align}\label{eqn:admnet}
\   c(\alpha) \cdot \val_2(\CL) &= c(\alpha) \cdot \max_{x \in X} \sum_{i=1}^l \alpha^{m(V_{i-1})}\cdot \diam(X) \cdot g\left(\frac{1}{\sigma(V_i)}\right) \notag \\
\ & \ge \max_{x \in X}  \sum_{i=1}^l  \alpha^{m(V_{i-1})}\cdot \diam(X) \cdot g\left(\frac{1}{\Psi(V_i)}\right) 
\end{align}
for a constant $c(\alpha)$. This follows from sub-multiplicativity of $g$ analogous to the proof of \lref{lem:chainingtree}.

Next we use the other direction of sub-multiplicativity that the Gaussian functional satisfies to argue that the right hand side of \eqref{eqn:admnet} is at least the value of the admissible net, up to constant factors. In particular, for any $x\in X$, consider the path from the root to the leaf $\{x\}$ in the labelled net $\CL$. For an integer $j \ge 0$, define $\CS_j = \CS_j(x)$ to be the set of vertices $V$ that satisfy $2^{2^j} \le \Psi(V) < 2^{2^{j+1}}$. Let us also define $m_j = m_j(x)= \max\{ m(V) \mid V \in \CS_j \}$ and let $J=J(x)$ be the maximum index $J$ such that the corresponding partition element indexed by $j$ that contains $x$, denoted $\CA_j(x)$ is the singleton set $\{x\}$. Note that for any vertex $V \in \CS_j$, we have that $g(1/\Psi(V)) \ge g(2^{-2^j})$. Therefore, denoting by $P(V)$ the parent of $V$ in the labelled net $\CL$, the right hand side of \eqref{eqn:admnet} for the given $x$ can be lower bounded as follows: 
\begin{align*}
 \sum_{i=1}^l  \alpha^{m(V_{i-1})}\cdot \diam(X) \cdot g\left(\frac{1}{\Psi(V_i)}\right) &=\diam(X) \sum_{j \in [J]} \sum_{V \in \CS_j} \alpha^{m(P(V))} g(2^{-2^j}) \\
\  &=\diam(X)\sum_{j \in [J]} g(2^{-2^j})\sum_{V \in \CS_j} \alpha^{m(P(V))} \\
\ & \ge \frac{\alpha}{2(1-\alpha)}  \cdot \diam(X) \sum_{j \in [J]} g(2^{-2^j}) \left(\alpha^{m_{j-1}} - \alpha^{m_j}\right) 
\end{align*}

Using that $g(2^{-2^j})  \ge C g(2^{-2^{j-1}})$ for $C>1$, and that $\diam(\CA_j(x)) \le \alpha^{m_{j-1}} \cdot \diam(X)$, by reindexing (note that $\diam(A_J(x))=0$), we see that
\begin{align*}
\ \diam(X) \sum_{j \in [J]} g(2^{-2^j}) \left(\alpha^{m_{j-1}} - \alpha^{m_j}\right) \geq (C-1) \sum_{j\in [J]} g(2^{-2^j}) \cdot \diam(\CA_j(x)).
\end{align*}
By taking the maximum over $x$, it follows from \eqref{eqn:admnet} and the above that $\val_2(\CL)\gtrsim \val_2(\CA)$.
\end{proof}

\end{document}